\documentclass{amsart}

\usepackage[utf8]{inputenc}
\usepackage[T1]{fontenc}

\usepackage{amssymb}
\usepackage{amsmath}
\usepackage{lmodern}
\usepackage[english]{babel}

\pdfoutput=1

\newtheorem{thm}[subsection]{Theorem}
\newtheorem{lem}[subsection]{Lemma}
\newtheorem{prop}[subsection]{Proposition}
\newtheorem{cor}[subsection]{Corollary}

\theoremstyle{definition}

\theoremstyle{remark}
\newtheorem{rem}[subsection]{Remark}

\numberwithin{equation}{subsection}

\usepackage
{hyperref}

\usepackage[
style=ieee-alphabetic,
giveninits=true,
isbn=false,
url=false,
]{biblatex}
\AtEveryBibitem{%
  \clearfield{note}%
}

\usepackage{dsfont}

\usepackage{xparse}
\usepackage{tikz-cd}
\usetikzlibrary{arrows}
\usepackage{comment}
\usepackage{mathtools}

\usepackage{enumitem}
\setlist[enumerate]{label=\upshape(\arabic*)}
\usepackage{multicol}

\def\<{\langle}
\def\>{\rangle}
\def\leq{\leqslant}
\def\geq{\geqslant}

\def\1{\mathds{1}}
\def\a{\alpha}

\def\b{\beta}
\def\C{\mathbb C}
\def\Cc{\mathcal C}

\def\e{\epsilon}

\def\F{\mathbb F}

\def\Ga{\Gamma}

\def\Jc{\mathcal J}
\def\K{\mathcal K}

\def\m{\mathfrak{m}}

\def\norm{\mathcal N}
\def\O{\mathcal O}
\def\P{\mathcal{P}}

\def\Pro{\mathbb P}
\def\Q{\mathbb Q}

\def\s{\sigma}

\def\Wc{\mathcal W}

\def\Z{\mathbb Z}

\newcommand{\diff}{\mathop{}\!\mathrm{d}}
\newcommand{\lege}[2]{%
  \genfrac{(}{)}{}{}{#1}{#2}%
}

\makeatletter
\newcommand{\extp}{\@ifnextchar^\@extp{\@extp^{\,}}}
\def\@extp^#1{\mathop{\bigwedge\nolimits^{\!#1}}}
\makeatother

\DeclareMathOperator{\ab}{ab}
\DeclareMathOperator{\aff}{aff}

\DeclareMathOperator{\Aut}{Aut}
\DeclareMathOperator{\can}{can}

\DeclareMathOperator{\Dic}{Dic}
\DeclareMathOperator{\disc}{disc}

\DeclareMathOperator{\et}{\'et}

\DeclareMathOperator{\Fix}{Fix}

\DeclareMathOperator{\Gal}{Gal}
\DeclareMathOperator{\GL}{GL}

\DeclareMathOperator{\Id}{Id}
\DeclareMathOperator{\id}{id}
\DeclareMathOperator{\Img}{Im}
\DeclareMathOperator{\Ind}{Ind}

\DeclareMathOperator{\Spec}{Spec}

\DeclareMathOperator{\tame}{t}

\DeclareMathOperator{\Tr}{Tr}
\DeclareMathOperator{\tr}{tr}

\DeclareMathOperator{\unr}{ur}

\DeclareMathOperator{\WD}{WD}

\DeclareMathOperator{\wild}{w}
\def\mydash{\protect\nobreakdash-\hspace{0pt}}

\DeclareDocumentCommand\bigslant{ m m g }{
{\IfNoValueT{#3}{\left.}\raisebox{.25em}{$#1$}\!\IfNoValueT{#3}{\middle/}\IfNoValueF{#3}{#3}\!\raisebox{-.25em}{$#2$}\IfNoValueT{#3}{\right.}}
}

\begin{document}

\title[Root numbers of curves of genus two with maximal ramification]{Root numbers of $5$-adic curves of genus two having maximal ramification}

\author{Lukas Melninkas}
\address{IRMA, UMR 7501, Université de Strasbourg et CNRS, 7 rue René Descartes, 67000 Strasbourg, France}
\curraddr{}
\email{melninkas@math.unistra.fr}
\thanks{}


\date{\today}

\maketitle
\vspace{-0.8cm}

\begin{abstract} 
The formulas for local root numbers of abelian varieties of dimension one are known. In this paper we treat the simplest unknown case in dimension two by considering a curve of genus 2 defined over a $5$-adic field such that the inertia acts on the first $\ell$-adic cohomology group through the largest possible finite quotient, isomorphic to $C_5\rtimes C_8$. We give a few criteria to identify such curves and prove a formula for their local root numbers in terms of invariants associated to a Weierstrass equation.
\end{abstract}

\section*{Introduction}

Given an abelian variety $A$ defined over a number field $\K$, its global root number $w(A/\K)$ is the sign appearing in the conjectural functional equation of its completed $L$-function. Granting the general Birch--Swinnerton-Dyer conjecture, ${w(A/\K)=-1}$ exactly when the Mordel--Weil rank is odd. Due to Deligne \autocite{deligne_eq_fonctionelle}, we can define $w(A/\K)$ unconditionally by computing the local root numbers $w(A_v/\K_v)$ of the completed abelian variety at each place $v$ of $\K$. 

For each infinite place we have $w(A_v/\K_v)=(-1)^{\dim A}$. If $A$ has good reduction at a finite place $v$, then $w(A_v/\K_v)=1$. This allows to define \[w(A/\K)=\prod_v w(A_v/\K_v),\] the product being taken over all places of $\K$. The local root numbers at places of bad reduction are signs $\pm1$ and are defined in a general way as we explain next. 

Let $p$ be a prime number, let $K/\Q_p$ be a finite extension with an algebraic closure $\overline{K}$, and let $A/K$ be an abelian variety. We choose another prime number $\ell\neq p$ and consider the $\ell$-adic Galois representation $\rho_\ell$ on the étale cohomology group $H^1_{\et}(A_{\overline{K}},\Q_\ell)$. Applying Grothendieck's monodromy construction we obtain a complex Weil--Deligne representation $\WD(\rho_\ell)$, whose isomorphism class does not depend on $\ell$ (see, e.g., \autocite[Cor.~1.15]{sabitova_root}). Next, following Deligne, after choosing an additive character $\psi$ on $K$ and a Haar measure $\diff x$ on $K$, we consider the $\e$-factor $\e(\WD(\rho_\ell),\psi,\diff x)\in\C^\times$. The local root number is then defined as \[w(A/K):=\frac{\e(\WD(\rho_\ell),\psi,\diff x)}{|\e(\WD(\rho_\ell),\psi,\diff x)|}.\] We note that $w(A/K)$ does not depend on $\ell$, $\psi$, or $\diff x$, see, e.g., \autocite[\S11,\S12]{rohrlich}

It follows (see, e.g., \autocite[Prop.~3.1]{chai_semiab}) from the semi-stable reduction theorems and the theory of $p$-adic uniformization that there exists an abelian variety $B/K$ with potentially good reduction and an extension $S$ of $B$ by a torus $T$ such that the rigid analytification of $A$ is a quotient of the analytification of $S$ by a lattice. Then, it follows from the result of Sabitova \autocite[Prop.~1.10]{sabitova_root} that $w(A/K)$ can be determined by computing $w(B/K)$ and the Galois action on $T$. In this paper we treat the case when $A/K$ itself has potentially good reduction. This condition is equivalent to $T=0$ and, by the criterion of Néron--Ogg--Shafarevich, to the condition that the image of inertia via $\rho_\ell$ is finite. By Serre--Tate \autocite[p.~497, Cor.~2]{serre_tate} the representation $\rho_\ell$ is at most tamely ramified whenever $p> 2\dim(A)+1$. 

If $A/K$ is an elliptic curve with potentially good reduction, formulas for root numbers have been given by Rohrlich \autocite{rohrlich_formulas} when $p\geq5$, by Kobayashi \autocite{kobayashi} when $p=3$, and by the Dokchitsers \autocite{dd_root_ellc2} when $p=2$. The case of Jacobians having semistable reduction have been studied by Brumer--Kramer--Sabitova \autocite{brumer_kramer_sabitova}. For general abelian varieties, the case when $\rho_\ell$ is tamely ramified has been studied by Bisatt \autocite{bisatt}. 

\subsection{The main setup and results} We consider a curve $C$ of genus $2$ defined over a $5$-adic field $K$. Let $J(C)/K$ be its Jacobian surface. Our aim is to produce a formula for $w(C/K):=w(J(C)/K)$ in terms of other invariants of $C/K$. We suppose that $J(C)/K$ has potentially good reduction and that the associated Galois representation $\rho_\ell$ is wildly ramified. We suppose further that $\rho_\ell$ has the maximal possible inertia image, isomorphic to the semi-direct product $C_5\rtimes C_8$ where $C_8$ acts on $C_5$ via $C_8\twoheadrightarrow \Aut(C_5)$. By choosing a Weierstrass equation we define a discriminant $\Delta\in K^\times$, whose class in $K^\times/(K^\times)^2$ does not depend on the choice of the equation, see \ref{subs:g2_var_change}. 

Let $k_K$ be the residue field of $K$. We denote by $\lege{\cdot}{k_K}$ the Legendre symbol on $k^\times_K$ and by $(\cdot,\cdot)_K$ the quadratic Hilbert symbol on $K^\times\times K^\times$. We consider $v_K$ the normalized valuation of $K$ such that $v_K(K^{\times})=\Z$. Let $F_5$ denote the Frobenius group on $5$ elements, defined as the semi-direct product $F_5=C_5\rtimes C_4$ where $C_4$ acts faithfully on $C_5$.

\begin{thm}[{Prop.~\ref{prop:hyperell_eq_spec}, Prop.~\ref{prop:hyperell_max_equiv}, Thm.~\ref{thm:max_ramif_rootN}}]\label{thm:max_g2_statement} 
Let $C/K$ be a smooth projective curve of genus $2$ defined over a $5$-adic field $K$. We suppose that the associated $\rho_\ell$ has finite inertia image of order divisible by $5$. There exists an equation $Y^2=P(X)$ defining $C/K$ with unitary, irreducible $P\in K[X]$ of degree $5$ having integral coefficients and a constant term $a_6$ of valuation prime to $5$.

The image of inertia of $\rho_\ell$ is the maximal possible, i.e. isomorphic to $C_5\rtimes C_8$, if and only if any of the following equivalent conditions is verified :
\begin{enumerate}
\item For any discriminant $\Delta$ of $C/K$ the valuation $v_K(\Delta)$ is odd; 
\item The $\F_2$-linear Galois representation on the $2$-torsion points $J(C)[2]$ has inertia image isomorphic to the Frobenius group $F_5$;
\item The Artin conductor $a(C/K)$ of $\rho_\ell$ is odd.
\end{enumerate}
In this case, the root number is given by \[w(C/K)=(-1)^{[k_K:\F_5]+1}\cdot \lege{v_K(a_6)}{k_K}\cdot(\Delta,a_6)_K.\]
\end{thm}

\begin{rem} The setting of Thm.~\ref{thm:max_g2_statement} is a particular case of the study by Coppola \autocite{coppola_max}, where a description of $\rho_\ell$ is given. Very recently, building on Coppola's results, Bisatt \autocite[Thm.~2.1]{bisatt_wild} produced similar formulas of root numbers of hyperelliptic curves.\end{rem}

\subsection{Structure of the paper} In Section~\ref{sect:wild_char_rootN} we recall some theory of $\e$-factors of one-dimensional Weil representations and give formulas for root numbers in some wild ramification cases by using explicit local class field theory. In Section~\ref{sect:wild_jac2} we present some properties of genus $2$ curves with wild ramification. In Section~\ref{sect:spec_f} we employ the Artin--Schreier theory in order to study $\rho_\ell$ via the automorphisms of curves over finite fields. In Section~\ref{sect:max_inertia} we prove a few characterizations of the maximal ramification case and exploit some of its implications. Section~\ref{sect:max_proof} is dedicated to proving the formula of Thm.~\ref{thm:max_g2_statement}, where we connect the results of Section~\ref{sect:wild_char_rootN} to a particular Weierstrass equation. 

\subsection*{Aknowledgements} I thank my thesis advisors Adriano Marmora and Rutger Noot as well as Kęstutis Česnavičius and Takeshi Saito for their remarks and suggestions concerning the manuscript. I also thank Jeff Yelton for answering my questions about his results on the splitting fields of the $4$-torsion of Jacobians. The study presented in this paper constitutes a part of my doctoral thesis. The results of Thm.~\ref{thm:max_g2_statement} have been obtained independently of the preprint \autocite{bisatt_wild}.

\section*{Notation and conventions}
Let $p$ be a prime number, and let $K/\Q_p$ be a finite extension. We adopt the convention that every algebraic extension of $K$ used in this text is a subfield of $\overline{K}$. We fix the following notation. 
\vspace{3mm}

{\centering
\begin{tabular}{cp{0.40\textwidth}|}\label{table:notation_p_adic}
$v_K$ & the valuation of $K$ normalized by $v_K(K^\times)=\Z$;\\
$\O_K$ & the ring of integers; \\
$\m_K$ & the maximal ideal;\\
$\varpi_K$ & a uniformizer;\\
$k_K$ & the residue field;\\
$q_K$ & the order $|k_K|$;\\
$\m^n_K$ & the subgroup $\varpi^n_K\O_K\subset K$ for any $n\in\Z$;\\
$U^n_K$ & $1\!+\!\m^n_K$ for $n\!\geq\!1$, and $U^0_K\!=\!\O_K^\times$; \\
$\lege{\cdot}{k_K}$ & the Legendre symbol on $k_K^\times$;\\
$(\cdot,\cdot)_K$ & the quadratic Hilbert symbol on $K^\times\times K^\times$;\\
\end{tabular}
\begin{tabular}{cp{0.35\textwidth}}
$\overline{K}$ & an algebraic closure of $K$;\\
$\overline{k}_K$ & the residue field of $\overline{K}$;\\
$\Ga_K$ & the group $\Gal(\overline{K}/K)$;\\
$W_K$ & the Weil subgroup of $\Ga_K$; \\
$I_K$ & the inertia subgroup; \\
$I_K^{\wild}$ & the wild inertia subgroup;\\
$\varphi_K$ & a lift in $W_K$ of the geometric Frobenius;\\
$\chi_{\unr}$ & the unramified (cy\-clo\-to\-mic) character $W_{\Q_p}\to \C^\times$ such that $\chi_{\unr}(\varphi_K)=q_K^{-1}$ for every finite $K/\Q_p$.\\ 
\end{tabular}}
\vspace{3mm}

By a \textit{Weil representation} on a complex vector space $V$ we mean a group homomorphism $\rho:W_K\to \GL(V)$ such that $\rho(I_K)$ is finite. For any $s\in\C^\times$, its Tate twist is $\rho(s):=\rho\otimes\chi_{\unr}^s$. 

Let $\theta_K: K^\times\cong W_K^{\ab}$ be Artin's reciprocity map normalized to send a uniformizer to the class of a geometric Frobenius lift. It follows that $||\cdot||_K:=\chi_{\unr}\circ\theta_K$ is the non-Archimedean norm on $K$ induced by $v_K$. For every finite Galois extension $L/K$, the map $\theta_K$ induces an isomorphism $\theta_{L/K}:K^\times/\norm_{L/K}(L^\times)\cong \Gal(L/K)^{\ab}$, where $\norm_{L/K}:L^\times \to K^\times$ is the norm map. 

Given schemes $X$, $S$, $S'$ as well as morphisms $X\to S$ and $S'\to S$, we will write $X_{S'}:=X\times_S S'$, and also $X_{R'}:=X\times_R R':=X_{S'}$ if $S'=\Spec R'$ and $S=\Spec R$ are affine. 

\section{Root numbers and explicit class field theory}\label{sect:wild_char_rootN}

Let $p>2$ be a prime number and let $K/\Q_p$ be a finite extension.

\subsection{Choice of an additive character}\label{subs:add_char_ex} By an \textit{additive character} we mean a locally constant group homomorphism $\psi:K\to\C^\times$. By $n(\psi)$ we denote the largest integer $n$ such that $\psi$ is trivial on $\m_K^{-n}$, called the \textit{level} of $\psi$. In order to simplify the computations of the root number we fix a particular character. Let $\psi_k$ be the composition \[\psi_k:\O_K\twoheadrightarrow k\xrightarrow{\tr_{k/\F_p}}\Z/p\Z\xhookrightarrow{\exp\left(\frac{2\pi i}{p}\cdot\right)}\C^\times.\] We see that $\psi_k$ is trivial on $\m_K$. Since $\C^\times$ is divisible, $\psi_k$ can be extended non-uniquely to an additive character of $K$, which we again denote by $\psi_k$. Since $\tr_{k/\F_p}$ is nontrivial, independently on the choice of the extension, we have $n(\psi_k)=-1$. Moreover, every additive character of level $-1$ is given by $x\mapsto\psi_k(cx)$ for some $c\in\O_K^\times$.

\subsection{$\psi$-gauges of Weil characters}\label{subs:gauges} Let $\chi:W_K\to\C^\times$ be a one-dimensional ramified Weil representation. We identify $\chi$ with a quasi-character of $K^\times$ via $\theta_K$. The Artin conductor $a(\chi)$ is the smallest integer $a$ such that $\chi$ is trivial on $U_K^a$. Let $n:=\big\lfloor\frac{a(\chi)+1}{2}\big\rfloor.$ The map $x\mapsto \chi(1+x)$ defined for $x\in\m_K^{n}$ is additive, trivial on $\m_K^{a(\chi)}$, and extends to an additive character $\psi_\chi$ of $K$ with $n(\psi_\chi)=a(\chi)$. Let $\varpi_K\in\m_K$ be a uniformizer. Then $x\mapsto \psi_\chi(\varpi_K^{a(\chi)-1}x)$ has level $-1$. Thus, there exists $c_\chi \in K^\times$, called a \textit{$\psi_k$-gauge of $\chi$}, of valuation $-a(\chi)+1$, unique modulo $\m_K^{-n+1}$, such that for all $x\in\m_K^n$, \begin{equation}\label{eq:gauge_def}\chi(1+x)=\psi_k(c_{\chi}x).\end{equation}

\subsection{Epsilon factors of characters}\label{subs:root_n_char} In addition to the setting of \ref{subs:gauges}, we fix a Haar mesure $\diff x$ on $K$. We recall from \autocite[(3.4.3.2)]{deligne_eq_fonctionelle} that the \textit{$\epsilon$-factor} of $\chi$ relative to $\psi_k$ and $\diff x$ is defined as the integral
\begin{equation}\label{eq:eps_def}\epsilon(\chi,\psi_k,\diff x):= \int_{\varpi_K^{-a(\chi)+1}\O_K^\times}\chi^{-1}(x)\psi_k(x)\diff x.\end{equation}
We will be mainly interested in the \textit{root number} \[w(\chi,\psi_k):=\frac{\epsilon(\chi,\psi_k,\diff x)}{|\epsilon(\chi,\psi_k,\diff x)|},\] which does not depend on $\diff x$. 

For $a,b\in\C^\times$ we will write $a\approx b$ whenever $ab^{-1}$ is contained in the subgroup of $\C^\times$ generated by positive real numbers and the complex roots of unity of $p$-power orders. We note that if $p\neq2$ and $a,b\in\{-1,1\}$ are such that $a\approx b$, then $a=b$.

\subsection{}\label{sect:weil_char_to_finite} The group $\chi(I_K)$ is finite and cyclic, we denote its order by $ep^r$ with $e$ prime to $p$. We view the restriction $\chi |_{I_K}$ as a character of the group $\Gal(K^{\ab}/K^{\unr})$. The group $\ker(\chi|_{I_K})$ cuts out an abelian extension $L'/K$ containing $K^{\unr}$. The closure of the subgroup generated by $\varphi_K$ in $\Gal(L'/K)$ cuts out a totally ramified abelian extension $L/K$, such that $L'=LK^{\unr}$. We then have canonical isomorphisms \[\bigslant{\Gal(K^{\ab}/K^{\unr})}{\ker(\chi|_{I_K})}\cong \Gal(L'/K^{\unr})\cong \Gal(L/K).\] The restriction $\chi|_{I_K}$ induces a faithful complex one-dimensional representation of the finite group $\Gal(L/K)$, and thus $\Gal(L/K)$ must be cyclic. Let $M/K$ be the unique subextension of $L/K$ of degree $p^r$. Then $\chi^e|_{I_K}$ has order $p^r$ and induces a faithful character of the cyclic group $\Gal(M/K)$.

The following is an amalgamation of some of the results of \autocite{kobayashi} and \autocite{abbes_saito}.

\begin{thm}\label{thm:eps_gauges} Let $\psi_k$ be as in \ref{subs:add_char_ex}. Let $\chi:W_K\to \C^\times$ be a Weil character such that $|\chi(I_K^{\wild})|=p$. Let $ep=|\chi(I_K)|$ with $e$ prime to $p$. Let $M/K$ be as in \ref{sect:weil_char_to_finite}. We denote by $\s\in\Gal(M/K)$ the generator that is sent to $\exp(\frac{2\pi i}{p})$ via $\chi$. Let $\varpi_M$ be a uniformizer of $M$, and let $\delta_\chi:= \norm_{M/K}(1-\frac{\s(\varpi_M)}{\varpi_M})$. Let us write $\delta_\chi=u\varpi_K^{v_K(\delta_\chi)}$ with $u\in\O_K^\times$, whose class in $k_K^\times$ we denote by $\overline{u}$.
\begin{enumerate}\item If $a(\chi)$ is even, then $\epsilon(\chi,\psi_k,\diff x)\approx \chi(\delta_\chi)$;
\item If $a(\chi)$ is odd, and $p\equiv 1\bmod 4$, then \[\epsilon(\chi,\psi_k,\diff x)\approx -\chi(\delta_\chi)\cdot\lege{2\overline{u}}{k_K}\cdot(-1)^{[k_K:\F_p]}.\]
\end{enumerate}
\end{thm}


\begin{lem}\label{lem:delta_gauge}
We have $v_K(\delta_\chi)=a(\chi)-1$ and $c_\chi\delta_\chi \in U^1_K$. In particular, $\chi^{-1}(c_\chi)\approx \chi(\delta_\chi)$.
\end{lem}

\begin{proof}
The lemma is essentially proved in \cite[p.~618]{kobayashi}. We adapt Kobayashi's argument in our setting.

Let $t$ be the largest integer such that the $t$-th ramification subgroup $G_t$ of $\Gal(M/K)$ is nontrivial. We then have $G^t=G_t=\Gal(M/K)$ and $G^{t'}=\{1\}$ for $t'>t$, see \autocite[V.\S3]{serre_localfields}. The reciprocity map (see \autocite[XV.\S2]{serre_localfields}) and $\chi$ induce a commutative diagram
\begin{equation} \begin{tikzcd}
\bigslant{U_K^t}{U_K^{t+1}\norm_{M/K}(U_M^t)} \arrow[r, "\sim"] & G^t= \Gal(M/K) \arrow[r, hookrightarrow ,"\chi^e|_{I_K}"] & \C^\times \\
U_K^t \arrow[u, two heads] \arrow[r, hookrightarrow] & \Gal(K^{\ab}/K^{\unr}) \arrow[u, two heads] \arrow[r, "\chi|_{I_K}"] & \arrow[u, "z\mapsto z^e"'] \C^\times.
\end{tikzcd}\end{equation}

As $e$ is prime to $p$ we observe that $a(\chi)=a(\chi^e)$ and that $a(\chi^e)=t+1$. Since $\s\in G_t\!\setminus\! G_{t+1}$, by using \autocite[IV.Prop.~5]{serre_localfields} we obtain \begin{equation}v_K(\delta_\chi)=v_M(1-\frac{\s(\varpi_M)}{\varpi_M})=t=a(\chi)-1.\end{equation}

Applying \autocite[XV.\S3,~Exercise~1]{serre_localfields} shows that for all $v\in U_K^t$,
\begin{equation}\label{eq:explicit_lcft} \theta_{M/K}(v) = \s^{\tr_{k_K/\F_p}\left((v-1)/\delta_\chi \bmod \m_K \right)}. \end{equation} For every $x\in\m_K^{a(\chi)-1}\subseteq\m_K^n$, taking the image of \eqref{eq:explicit_lcft} by $\chi^e$, we obtain
$\chi^e(1+x)=\psi_k(e\delta_\chi^{-1}x)$, and taking the $e$-th power of \eqref{eq:gauge_def} gives $\chi^e(1+x)=\psi_k(ec_\chi x)$. We note that $e\delta_\chi^{-1}\m_K^{a(\chi)-1}=\O_K$. Thus, \[\psi_k((1-\delta_\chi c_\chi)\O_K)=\psi_k((e\delta_{\chi}^{-1}-ec_\chi)\m_K^{a(\chi)-1})=1.\] Therefore, we must have $1-\delta_\chi c_\chi\in\m^{-n(\psi_k)}_K=\m_K$. The last part of the lemma follows from the fact that $\chi(U_K^1)$ is a finite $p$-group.
\end{proof}

\begin{proof}[Proof of Theorem~\ref{thm:eps_gauges}]
We shall apply \autocite[Prop.~8.7,~(ii)]{abbes_saito} which allows to express the epsilon factor using a refined $\psi_k$-gauge $c$ of $\chi$. Abbes--Saito proves that there exists an element $c\in K$, unique modulo $\m_K^{-n+1}$, such that for every $x\in\m_K^{a(\chi)-n}$ we have \[\chi\left(1+x+{\textstyle \frac{x^2}{2}}\right)=\psi_k(cx).\] Let $\tau:k_K\to K$ be the Teichmüler lift. We consider the quadratic Gauss sum \[G_{\psi_k}:=\sum_{x\in k_K}\psi_k(\tau(x)^2).\] The formulas $G_{\psi_k}=\sum_{x\in k_K^\times}\lege{x}{k_K}\psi_k(\tau(x))$ and $G_{\psi_k}^2= \lege{-1}{k_K}q_K$ are well-known (see, e.g., \autocite[\S1.1]{gauss_sums_book}). The Abbes--Saito formula \autocite[\nopp{}(8.7.3)]{abbes_saito} can be rewritten as
\begin{equation}\label{eq:abbes_saito}\epsilon(\chi,\psi_k,\diff x)\approx \chi^{-1}(c)\psi_k(c)\lege{-1}{k_K}^{\binom{a(\chi)}{2}}\!\!G_{\psi_k}^{-a(\chi)}\times \begin{cases}1 \!&\text{if $a(\chi)$ is even,}\\ (\!-2c,\varpi_K)_K \!&\text{if $a(\chi)$ is odd}.\end{cases}\end{equation}
Since $c$ is also a $\psi_k$-gauge of $\chi$, we have $\chi^{-1}(c)\approx\chi(\delta_\chi)$ by Lemma~\ref{lem:delta_gauge}. For $r\in\Z$ large enough, $\psi_k(p^rc)=1$, so $\psi_k(c)\approx1$. If $a(\chi)$ is even, then it is straightforward to verify that (1) holds.

We assume the hypotheses of (2). Then $\lege{-1}{k_K}=1$, and $G_{\psi_k}\approx-(-1)^{[k_K:\F_p]}$, see \autocite[Thm.~11.5.4]{gauss_sums_book}. We also have $(-2,\varpi_K)_K=\lege{2}{k_K}$. Taking into account Lemma~\ref{lem:delta_gauge} and making the relevant substitutions into \eqref{eq:abbes_saito} we are left to prove that $(c,\varpi_K)_K=\lege{\overline{u}}{k_K}$. Lemma~\ref{lem:delta_gauge} also shows that $c\in u^{-1}\varpi_K^{-a(\chi)+1}U^1_K$. Since $a(\chi)$ is odd and $U^1_K$ is pro-$p$, the Hilbert symbol is trivial on $\varpi_K^{-a(\chi)+1}U^1_K$, thus \[(c,\varpi_K)_K=(u,\varpi_K)_K=\lege{\overline{u}}{k_K}.\qedhere\]
\end{proof}

\begin{prop}\label{prop:kobayashi_D_delta}
We continue in the situation of Thm.~\ref{thm:eps_gauges}. Let $\a\in \O_M$ be such that $p\nmid v_M(\a)$, and let $D_{\a,\chi}:=\norm_{M/K}\left(1-\frac{\s(\a)}{\a}\right)$. Then \[D_{\a,\chi}\equiv v_M(\a)\delta_\chi\bmod U_K^1.\] \end{prop}

\begin{proof}
A detailed proof when $p=3$ can be found in \autocite[p.~614]{kobayashi}. It generalizes for any $p>2$ without significant modifications.
\end{proof}

\begin{cor}\label{cor:char_rootN_D}
If $a(\chi)$ is even and $\a\in\O_M$ is such that $p\nmid v_M(\a)$, then \[\epsilon(\chi,\psi_k,\diff x)\approx\chi\left(\frac{D_{\a,\chi}}{v_M(\a)}\right).\] 
\end{cor}

\begin{proof}
Follows from Thm.~\ref{thm:eps_gauges}.(1) and Prop.~\ref{prop:kobayashi_D_delta}.
\end{proof}

\section{Curves of genus $2$ and wild ramification}\label{sect:wild_jac2}

\subsection{Generalities}\label{subs:curve_g2} Let $K$ be a $p$-adic local field with $p\neq2$, and let $C/K$ be a smooth, projective, and geometrically connected curve of genus 2 defined over $K$. The curve $C/K$ is hyperelliptic (see \autocite[7.~Prop.~4.9]{liu_book}), and there exists an open affine subvariety $C_{\aff}$ of $C$ which is defined by a single Weierstrass equation
\begin{equation}\label{eq:hyperell_eq}
Y^2=P(X),
\end{equation}
where $P\in K[X]$ has simple roots and has degree $5$ or $6$.


\subsection{Discriminants} Following \autocite[\S2]{liu_models} we define the \textit{discriminant} of an equation \eqref{eq:hyperell_eq} in terms of the discriminant of the polynomial $P$: let $a_0$ be the leading coefficient of $4P$, then
\begin{equation}\label{eq:disc_def}
\Delta(P):=\begin{cases} 2^{-12}\disc(4P) & \text{if }\deg P=6,\\ 2^{-12}a_0^2\disc(4P)& \text{if } \deg P=5.\end{cases}
\end{equation}
In particular, if $P(X)=(X-\a_1)(X-\a_2)(X-\a_3)(X-\a_4)(X-\a_5)$, then 
\begin{equation}\label{eq:disc_unitary}
\Delta(P)=2^8\prod_{1\leq i<j\leq 5}(\a_i-\a_j)^2.
\end{equation}
We note that $\Delta(P)\neq0$ since $C$ is non-singular.

\subsection{Change of variables}\label{subs:g2_var_change} The equation \eqref{eq:hyperell_eq} is unique up to a transformation
\begin{equation}\label{eq:hyperell_change_var}
X'=\frac{aX+b}{cX+d}, \hspace{5mm} Y'=\frac{eY}{(cX+d)^3}\end{equation}
where
$\begin{pmatrix}a&b \\ c&d \end{pmatrix}\in\GL_2(K)$ and $e\in K^\times.$ If $Y'^2=P'(X')$ is obtained from \eqref{eq:hyperell_eq} via \eqref{eq:hyperell_change_var}, then 
the new discriminant is 
\begin{equation}\label{eq:var_change_disc}\Delta(P')=e^{20}\begin{vmatrix}a&b\\c&d\end{vmatrix}^{-30}\! \Delta(P) \, .\end{equation} As an immediate consequence, the class of a discriminant in $K^\times/(K^\times)^2$ does not depend on the choice of a Weierstrass equation.

\subsection{Root numbers and semi-stable reduction} Let $C/K$ be as in \ref{subs:curve_g2} and let $J(C)/K$ be its Jacobian. We denote by $\Jc(C/K)_k^\circ$ the neutral component of the special fiber of the Néron model of $J(C)/K$. Let $\ell\neq p$ be a prime number. We have an isomorphism of $\ell$-adic $\Ga_K$-representations \[H^1_{\et}(C_{\overline{K}},\Q_\ell)\cong H^1_{\et}(J(C)_{\overline{K}},\Q_\ell),\] we denote either of them by $\rho_\ell$. The root number $w(C/K):=w(\rho_\ell)$ is defined via the the complex Weil--Deligne representation associated to $\rho_\ell$ (see, e.g., \autocite{rohrlich}). Let $L/K$ be a finite extension over which $C$ has stable reduction. Then $J(C_L)/L$ has semi-stable reduction, i.e. that $\Jc(C_L/L)_k^\circ$ is an extension of an abelian variety by a torus. Using Sabitova's decomposition \autocite[Prop.~1.10]{sabitova_root} we can separate the contributions to $w(C/K)$ coming from the abelian and the toric parts of $\Jc(C_L/L)_{k_L}^\circ$. 

\subsection{Hypotheses} We suppose from now on that $J(C_L)/L$ has good reduction or, equivalently, that $J(C)/K$ has potentially good reduction. It follows from the semi-stable reduction theorems that this happens exactly when $\rho_\ell(I_K)$ is finite. In this case we write $|\rho_\ell(I_K)|=ep^r$ with $e$ coprime to $p$. We further suppose that $r\geq1$, i.e. $\rho_\ell$ is wildly ramified. Due to Serre--Tate \autocite[p.~497,~Cor.~2]{serre_tate}, necessarily, $p\leq5$. We will later suppose that $p=5$.

\subsection{Inertially minimal extensions}\label{subs:im_ext} It follows from the Néron--Ogg--Shafarevich criterion that $J(C)$ attains good reduction  over $L':=\overline{K}^{\ker \rho_\ell|_{I_K}}$ and that $L'/K^{\unr}$ is the minimal such extension. We call an algebraic extension $L/K$ \textit{inertially minimal (IM) for $J(C)/K$} if $I_L=I_{L'}= \ker \rho_\ell|_{I_K}$. In other words, $L/K$ is IM if and only if $J(C)$ has good reduction over $L$ and has bad reduction over every proper subextension of $K^{\unr}L/K^{\unr}$. 

\subsection{Good reduction and torsion}\label{subs:red_torsion} For $m\geq1$ we denote by $J(C)[m]$ the subgroup of $m$-torsion points of $J(C)(\overline{K})$ and by $K(J(C)[m])$ the smallest extension of $K$ over which all the points of $J(C)[m]$ are rational. For $m\geq 3$ coprime to $p$, it follows from Serre--Tate \autocite[Cor.~3, p.~498]{serre_tate} that the extension $K(J(C)[m])/K$ is IM for $J(C)/K$. Similarly, for $p\neq 2$, Serre~\autocite{serre_2torsion} shows that $|\rho_\ell(I_{K(J(C)[2])})|\leq2$. Thus, if $K(J(C)[2])/K$ is not an IM extension, then there is a totally ramified quadratic extension $L/K(J(C)[2])$ such that $L/K$ is IM for $J(C)/K$. Therefore, if $p\neq 2$, then the groups $\rho_\ell(I^{\wild}_K)$ and $I^{\wild}(K(J(C)[2])/K)$ are isomorphic. 

\subsection{We suppose for the rest of the section that $p=5$.}\label{subs:inertia_class} 
The groups $\rho_\ell(I_K)$ associated to abelian surfaces have been classified by Silverberg--Zarhin \autocite[Thm.~1.7]{silverberg_zarhin}. In our case $\rho_\ell(I_K)$ is one of the four groups (in the notation of \autocite{group_names}) satisfying the inclusions $C_5\subset C_{10}\subset \Dic_5\subset C_5\rtimes C_8$. More precisely, the group $\rho_\ell(I_K)$ has the form $C_5\rtimes C_{2^i}$ where $C_{2^i}$ is a subgroup of $C_8$ acting on $C_5$ with kernel $C_2\cap C_{2^i}\subset C_8$. 

Recall the Frobenius group $F_5=C_5\rtimes C_4$ where $C_4$ acts faithfully on $C_5$. In particular, $F_5$ and $\Dic_5$ are not isomorphic.

\begin{prop}\label{prop:wild5_good_equiv}
Suppose that $\rho_\ell$ is wildly ramified and let $L/K$ be a finite extension. If $J(C)$ has semi-stable reduction over $L$, then $C$ has good reduction over $L$, i.e. the minimal regular model $\Cc'/\O_L$ is smooth.
\end{prop}

\begin{proof}
From classical theorems we know that $C$ has semi-stable reduction over $L$, and that there exists a stable (flat) model $\Cc'_{\can}/\O_L$ of $C_L/C$. The ring $R:=\O_{K^{\unr}L}$ is strictly Henselian and $(\Cc' _{\can})_R$ is a stable model of $C_{K^{\unr}L}/K^{\unr}L$. Wild ramification of $\rho_\ell$ implies that $5$ divides $[K^{\unr}L:K^{\unr}]$. By studying the possible orders of automorphisms of stable curves Liu \autocite[Cor.~4.1.(4)]{liu_stable_red} showed that $(\Cc'_{\can})_{\overline{k}_L}/{\overline{k}_L}$ must be smooth. It follows that $(\Cc' _{\can})_R/R$ and hence $\Cc' _{\can}/\O_L$ are smooth. We may use \autocite[10.~Prop.~1.21]{liu_book} to conclude that $\Cc'/\O_L$ is smooth.
\end{proof}

\begin{rem}\label{rem:good_r_curve_jac} 
The hypotheses that $\rho_\ell$ is wildly ramified and that $K$ is $5$-adic are essential. The curve $C_L/L$ might have bad reduction even if $J(C)$ has good reduction over $L$. On the other hand, \autocite[Example~8, p.~246]{BLRNeron} shows that the non-rational irreducible components of $\Cc'_{\overline{k}_L}$ correspond to nontrivial abelian varieties as quotients of $\Jc(C_L/L)^\circ_{\overline{k}_L}$. Using this it can be shown in general that if $\Jc(C_L/L)^\circ_{\overline{k}_L}$ is a simple abelian variety, then $C_L/L$ has good reduction. 
\end{rem}

\subsection{An explicit IM extension}\label{subs:A_5_J_10} Let $Y^2=P(X)$ be a hyperelliptic equation defining $C/K$. Generalizing the results of Kraus~\autocite{kraus}, Liu \autocite[\S5.1]{liu_algo} provides an explicit description in terms of invariants of $P$ of the tame part of the minimal extension $L'/K^{\unr}$ over which $C$ has stable reduction. By Prop.~\ref{prop:wild5_good_equiv}, this extension is precisely the IM extension for $J(C)/K$ defined in~\ref{subs:im_ext}. Liu~\autocite[\S2.1]{liu_algo} defines a so-called affine invariant $A_5$. After Prop.~\ref{prop:hyperell_eq_spec} we will always have $A_5=1$. We fix an $8$th root of \[\beta:=-A_5^{-6}\Delta(P)\] in $\overline{K}$, which we denote by $\beta^{1/8}$, and let $\beta^{1/4}:=(\beta^{1/8})^2$. Let $L'_{\tame}/K^{\unr}$ be the maximal tamely ramified subextension of $L'/K^{\unr}$, then $L'/L'_{\tame}$ is totally wildly ramified of degree $5$. Liu proves that 
\begin{equation}\label{eq:liu_tame} L'_{\tame}=K^{\unr}(\beta^{1/8}).\end{equation} Let $\nu:=v_K(\beta)$, $M:=K(J(C)[2])$, $N:=K\left(\beta^{1/8}\right)$, $H:=K\left(\beta^{1/4}\right)$, and $L:=MN$. We fix a primitive $8$th root of unity $\zeta_8\in\overline{K}$.

Let us recall from \autocite[3.39,~Cor.~2.11]{mumford_tata2} that $M$ is the splitting field of $P$. The extension $M/K$ is finite Galois. The extension $L/K$ is finite but not necessarily Galois.

\begin{prop}\label{prop:explicit_IM}
The extension $L/K$ is IM for $J(C)/K$ or, equivalently, $L'=K^{\unr}L$. The extension $L(\zeta_8)/K$ is Galois. In particular, if the residual degree $f(K/\Q_5)$ is even, then $L/K$ is Galois. The extension $L/H$ is always Galois.
\end{prop}

\begin{proof}
From \ref{subs:red_torsion} we have $|\rho_\ell(I_M)|\leq 2$, so $\rho_\ell|_{\Ga_M}$ is at most tamely ramified. It now follows from \eqref{eq:liu_tame} that $J(C)$ has good reduction over $LK^{\unr}$, so $L'\subset LK^{\unr}$. On the other hand, we have $L'\supset L_{\tame}'= NK^{\unr}$, and $I_{L'}$ acts trivially on $J(C)[2]$ by the Néron--Ogg--Shafarevich criterion, so $L'\supset LK^{\unr}$. Therefore, $I_{L'}= I_L$.

The Galois closure of $N/K$ is $N(\zeta_8)/K$, so $L(\zeta_8)/K$ is Galois. Since $\Q_5$ contains the $4$th roots of unity and no primitive $8$th roots of unity, $K$ contains $\zeta_8$ if and only if $f(K/\Q_5)$ is even.

The extension $HM/K$ is Galois since $H/K$ and $M/K$ are such. Thus, $L/H$ is Galois as the compositum of $HM/H$ and $N/H$.
\end{proof}

\begin{prop}\label{prop:hyperell_irred_5}
The group $\Gal(M/K)$ is isomorphic to a subgroup of $F_5$ (as in \ref{subs:inertia_class}). As a consequence, the polynomial $P$ has an irreducible factor over $K$ of degree 5. 
\end{prop}

\begin{proof}
We recall that $\deg P=5$ or $6$, so we may view $\Gal(M/K)=\Gal(P)$ as a subgroup of $S_5$ or $S_6$, respectively. Since the wild inertia subgroup of $\Gal(P)$ is normal of order $5$, the group $\Gal(P)$ must be a subgroup of a normalizer subgroup $G$ of a $5$-cycle in $S_5$ or $S_6$. We naturally have $F_5\subseteq G$ and, in fact, an equality holds because for $n=5,6$ we have \[|G|=\frac{|S_n|}{\#\{5\text{-Sylow's in }S_n\}}=\frac{n!}{\frac{n!}{4\cdot5\cdot (n-5)!}}=20.\] 

If $P$ was irreducible over $K$ and had degree $6$, then $\Gal(P)$ would have a subgroup of index $6$, which is impossible. On the other hand, since $\Gal(P)$ contains a $5$-cycle, $P$ must have an irreducible factor of degree at least $5$.
\end{proof}

\begin{prop}\label{prop:inertia_disc_val} The group $\rho_\ell(I_K)$ is isomorphic to $C_5\rtimes C_8$, $\Dic_5$, $C_{10}$, or $C_5$ if and only if $\nu\equiv1\bmod 2$, $\nu\equiv2\bmod 4$, $\nu\equiv 4\bmod 8$, or $\nu\equiv0\bmod 8$, respectively. In particular, by denoting the ramification index of $L/K$ by $e(L/K)$ we have $40\mid e(L/K)\cdot\nu$. 
\end{prop}

\begin{proof}
By \eqref{eq:liu_tame} and Prop.~\ref{prop:explicit_IM}, the tame ramification index of $L/K$ is determined by the the residue $\nu\bmod 8$ and is exactly the maximal prime-to-$5$ divisor of $|\rho_\ell(I_K)|$. The group $\rho_\ell(I_K)$ can then be identified from the classification in \ref{subs:inertia_class}.
\end{proof}

\begin{prop}\label{prop:almost_abelian}
Let $\s\in I_K^{\wild}$ and let $\tau\in \Ga_K$ denote a lift of a topological generator of the tame inertia group $I_K^{\tame}$. Let $\varphi_L\in \Ga_L$ and $\varphi_{L(\zeta_8)}\in \Ga_{L(\zeta_8)}$ be lifts of the geometric Frobenii. Then: 
\begin{enumerate} 
\item The images $\rho_\ell(\s)$, $\rho_\ell(\tau^4)$, and $\rho_\ell(\varphi_L)$ commute;
\item The images $\rho_\ell(\tau)$ and $\rho_\ell(\varphi_{L(\zeta_8)})$ commute.
\end{enumerate}
In particular, $\rho_\ell(\varphi_{L(\zeta_8)})$ is central in $\rho_\ell(\Ga_K).$
\end{prop}

\begin{proof}
Prop.~\ref{prop:explicit_IM} shows that $\rho_\ell|_{I_L}$ is trivial and that $L'=LK^{\unr}$. Thus, for (1) we only need to show that the classes $\s I_L$, $\tau^4 I_L$, and $\varphi_L I_L $ in $\Gal(L'/K)=\Ga_K/I_L$ commute. From \ref{subs:inertia_class} we have $\s^5\in I_L$ and $\tau^8\in I_L$. We note that the subfield of $L'$ fixed by $\varphi_LI_L$ is $L$.


Let $F/K$ be the subextension of $M/K$ fixed by the unique $5$-Sylow subgroup of $\Gal(M/K)$. We claim that $L/F$ is abelian. We observe that $L/F$ is the compositum of the $C_5$-extension $M/F$ and the maximal at most tamely ramified subextension $L_{\tame}/F$ of $L/F$. By \ref{subs:red_torsion}, the ramification index of $L_{\tame}/F$ is at most $2$, so $L_{\tame}/F$ must be abelian. It follows that $L/F$ is abelian. The extension $L'/F$ is abelian as the compositum of $L/F$ and $K^{\unr}F/F$. 

We observe that the closure of the subgroup generated by $\tau^4I_K^{\wild}$ in $I_K$ cuts out the unique extension of $K^{\unr}$ of degree 4, which contains $F$ (we have $[F:K]\mid 4$ by Prop.~\ref{prop:hyperell_irred_5}). Thus, the class $\tau^4I_L$ is in $\Gal(L'/F)$. On the other hand, $\s I_L$ and $\varphi_L I_L$ are also in $\Gal(L'/F)$, so they all commute. 

For (2) we first note that, for $\gamma\in I_K$, we have $\eta:=\gamma\varphi_{L(\zeta_8)}\gamma^{-1}\varphi_{L(\zeta_8)}^{-1}\in I_K$. Since $L(\zeta_8)/K$ is Galois by Prop.~\ref{prop:explicit_IM}, we have $\gamma\varphi_{L(\zeta_8)}\gamma^{-1}\in\Ga_{L(\zeta_8)}$, and thus $\eta\in \Ga_{L(\zeta_8)}\cap I_K=I_{L}$. Thus, $\rho_\ell(\eta)$ is trivial, and hence (2) holds.
\end{proof}

\subsection{Particular form of hyperelliptic equation}\label{subs:hyperell_eq_spec} If $\deg P=6$, then Prop.~\ref{prop:hyperell_irred_5} shows that $P$ has a root in $K$. By applying a change of variables \eqref{eq:hyperell_change_var} that sends this root to the point at infinity, we may assume that the curve $C/K$ is defined by a Weierstrass equation $Y^2=P(X)$ with $P$ unitary and irreducible of degree $5$. Applying Liu's results \autocite[Prop.~5.1]{liu_algo} gives the following
\begin{prop}\label{prop:hyperell_eq_spec}
There exists an equation \[Y^2=X^5+a_2X^4+\ldots +a_6,\] which defines $C/K$ with $a_2,\ldots,a_6\in\O_K$ such that $v_K(a_6)\in\{1,2,3,4,6,7,8,9\}$. 
With respect to this equation, we have $A_5=1$.
\end{prop}

\begin{rem}
Prop.~\ref{prop:hyperell_eq_spec} is an analogue of Tate's algorithm for elliptic curves. Indeed, Liu also proved that the integer $v_K(a_6)$ determines the geometric type of the minimal regular model of $C/K$. These correspond to the Namikawa--Ueno \autocite{namikawa_ueno} types [VIII-i] and [IX-i] for $i=1,2,3,4$.
\end{rem}

\section{Galois action on the special fiber}\label{sect:spec_f} 

Let $k$ be a finite field of some characteristic $p>2$. We denote by $k(y)$ the field of rational functions in one variable over $k$.

\subsection{Artin--Schreier curves}\label{subs:artin_schreier_curves} We briefly recall some basic Artin--Schreier theory. If $f\in k(y)$ is not in the image of the map $g\mapsto g^p-g$, then the equation $x^p-x=f$ defines a smooth projective curve $C_f$ over $k$ together with a finite morphism $\pi:C_f\to\Pro^1_{k}$ of degree $p$. In other words, the function field $k(C_f)=k(x,y)$ is a cyclic extension of $k(y)$ of degree $p$. Inversely, every cyclic extension of $k(y)$ of order $p$ is of this form. We may assume that $f$ is in \textit{standard form}, i.e. each pole of $f$ is of order prime to $p$, this is well-known, see \autocite[\S2]{hesse_AS}. Then, $f$ has poles at exactly the branch points of $\pi$. In particular, if $\pi$ has a unique branch point in $\Pro_k^1(k)$ which is the pole of $y$, then $f$ is a polynomial in $y$. If this is the case, then the genus of $C_f$ is given by $g(C_f)=\frac{(\deg f-1)(p-1)}{2}$, see, e.g., \autocite[3.7.8.(d)]{stichtenoth_ff}.

For every $a\in k^{\times}$ and $c\in k$ we denote by $C_{a,c}$ the Artin--Schreier curve given by the equation $x^p-x-c=ay^2$.

\begin{lem}\label{lem:traces_C10}
Let $p\equiv1\bmod 4$. On the curve $C_{1,0}/\F_p$ we have the automorphisms $\s_1:(x,y)\mapsto(x+1,y)$, $\iota:(x,y)\mapsto(x,-y)$, and the endomorphism $F:(x,y)\mapsto(x^p,y^p)$. They commute pairwise and, for all $n,r,f\in\Z$, the trace of the pullback $(\iota^n\circ\s_1^r\circ F^f)^*$ on $H^1_ {\et}((C_{1,0})_{\overline{\F}_p},\Q_\ell)$ is given by \[\Tr (\iota^n\circ\s_1^r\circ F^f)^*= \begin{cases} (-1)^{n+1}p^{f/2} & \text{ if $f$ is even and $p\nmid r$}, \\ (-1)^np^{f/2}(p-1)&\text{ if $f$ is even and $p\mid r$}, \\ (-1)^{n+1}\lege{r}{\F_ p}p^{\frac{f+1}{2}}& \text{ if $f$ is odd.}\end{cases}\]
\end{lem}

\begin{proof}
It is straightforward to verify that $\s_1$, $F$, and $\iota$ commute. The hyperelliptic involution $\iota$ acts as multiplication by $-1$ on the Jacobian variety, so $(\iota^n)^*=(-\Id)^n$.

The curve $C_{1,0}$ has genus $\frac{p-1}{2}$, and $\dim H^1_{\et}((C_{1,0})_{\overline{\F}_p},\Q_\ell)=p-1$. We recall from the classical theory that the action of $F^*$  is semisimple and its eigenvalues have absolute value $\sqrt{p}$. We claim that $(F^2)^*$ acts as mulplication by $p$. For this we only need to show $\Tr(F^2)^*=p(p-1)$. The Lefschetz trace formula \[\Tr(F^2)^* =1+p^2-| C_{1,0}(\F_{p^2})|\] leaves to prove $| C_{1,0}(\F_{p^2})|=p+1$. For every $x\in\F_{p^2}$ we have $\Tr_{\F_{p^2}/\F_p}(x^p)=\Tr_{\F_{p^2}/\F_p}(x)$. It follows that every affine point $(x,y)\in C_{1,0}(\F_{p^2})$ must be such that $\Tr_{\F_{p^2}/\F_p}(y^2)=0$, which is equivalent to $y^2+y^{2p}=0$. The non-zero solutions of the latter satisfy $y^{2(p-1)}=-1$, raising this to the odd power $\frac{p+1}{2}$ leads to a contradiction. Thus, the affine points of $C_{1,0}(\F_{p^2})$ are $(x,0)$ with $x\in\F_{p}$, and thus the claim holds.

The polynomial $X^2-p$ is irreducible over $\Z$. Since the characteristic polynomial of $F^*$ is in $\Z[X]$, it must be $(X^2-p)^{\frac{p-1}{2}}$. 

We have $0=(\s_1^p)^*-\Id=(\s_1^*-\Id)\Phi_p(\s_1^*)$ where $\Phi_p\in\Z[X]$ is the $p$-th cyclotomic polynomial. The characteristic polynomial $P_{\s_1}$ of $\s_1^*$ is in $\Z[X]$, so its unitary irreducible divisors can only be $X-1$ or $\Phi_p$. Since $\deg P_{\s_1}=p-1$, we must have $P_{\s_1}(X)=(X-1)^{p-1}$ or $P_{\s_1}=\Phi_p$. The first case is impossible since $\s_1^*$ is nontrivial. Thus, $\Tr(\s_1^r)^*=-1$ if $r$ is prime to $p$, and $\Tr(\s_1^r)^*=p-1$ otherwise. The formulas for the case when $f$ is even hence follow.

If $f$ is odd, then
\begin{equation}\label{eq:trace_f_odd}
\Tr (\iota^n\circ\s_1^r\circ F^f)^*=(-1)^np^{\frac{f-1}{2}}\Tr(\s_1^r\circ F)^*. \end{equation} We use the Lefchetz formula \[\Tr(\s_1^r\circ F)^*=1+p-\left|\Fix(\s_1^r\circ F)\right|.\] The affine points $(x,y)\in C_{1,0}(\overline{\F}_p)$ fixed by $\s_1^r\circ F$ satisfy $x=x^p+r$ and $y=y^p$, so $y\in\F_p$ and $-r=x^p-x=y^2$. The latter equation has exactly $\lege{-r}{\F_p}+1=\lege{r}{\F_p}+1$ solutions in $y$ for each $r\in\F_p$. Each solution $y$ gives exactly $p$ solutions for $x^p-x=y^2$. We have proved that $\s_1^r\circ F$ has exactly $p\big(\lege{r}{\F_p}+1\big)+1$ fixed points, so 
\begin{equation}\label{eq:trace_a_frob}
\Tr(\s_1^r\circ F)^*=-\lege{r}{\F_p}p.
\end{equation}
Substituting \eqref{eq:trace_a_frob} into \eqref{eq:trace_f_odd} finishes the proof. 
\end{proof}

\begin{rem}
If $p\equiv 3\bmod4$, then using similar methods one can compute the order $|C_{1,0}(\F_{p^2})|=p(2p-1)+1$ and prove that $(F^*)^2+p=0$. Consequently, analogous formulas for the traces can be given.
\end{rem}

\subsection{The 5-adic setting}\label{subs:5-adic_wild} We now continue in the situation where $K$ is $5$-adic and $C/K$ is a curve of genus $2$ whose $\ell$-adic representation has cyclic wild inertia image $\rho_\ell(I_K^{\wild})$ of order $5$. Recall the notation of \ref{subs:A_5_J_10}.

\subsection{Galois action on the smooth model}\label{subs:extend_galois_model} By Prop.~\ref{prop:explicit_IM} and Prop.~\ref{prop:wild5_good_equiv}, the curve $C_{L}/L$ has good reduction, so its minimal regular model $\Cc'/\O_{L}$ is smooth. For every finite Galois extension $K'/K$ containing $L$, the minimal regular model of $C_{K'}/K'$ is given by the base change $\Cc'_{\O_{K'}}:=\Cc'\times_{\O_L}\O_{K'}$. Every element of $\Gal(K'/K)$ gives an $K'$-semilinear automorphism $C_{K'}\xrightarrow{\sim} C_{K'}$, which extends uniquely to an $\O_{K'}$-semilinear automorphism $\Cc'_{\O_{K'}}\xrightarrow{\sim} \Cc'_{\O_{K'}}$ (see, e.g., \autocite[Corollary~1.2]{liu_neron}). Passing to the projective limit shows that each $\gamma\in\Ga_K$ induces an $\O_{\overline{K}}$-semilinear automorphism $\gamma_{\Cc'}:\Cc'_{\O_{\overline{K}}}\xrightarrow{\sim} \Cc'_{\O_{\overline{K}}}$. The morphism preserves the special fiber, so we obtain a $\overline{k}_K$-semilinear $\Ga_K$-action on $\Cc'_{\overline{k}_K}$.

By functoriality, $\Ga_K$ acts on $H^1_{\et}(\Cc'_{\overline{k}_K},\Q_\ell)$, and, for every $n\in\Z$ prime to $p$, the smooth base change theorem provides an isomorphism of $\Ga_K$-modules 
\begin{equation}\label{eq:et_coh_fiber_iso}
H^1_{\et}(C_{\overline{K}},\Z/n\Z)\cong H^1_{\et}(\Cc'_{\overline{k}_{K}},\Z/n\Z).\end{equation}

We note that every element $\gamma\in\Ga_L$ acts on $\Cc'_{\O_{\overline{K}}}=\Cc'\times_{\O_L}\O_{\overline{K}}$ as $\id \times\gamma$. Since $I_K$ acts trivially on $\overline{k}_K$, the group $I_L$ acts trivially on $\Cc'_{\overline{k}_K}$, thus inducing an action of $I_K/I_L$ on $\Cc'_{\overline{k}_K}$ by $\overline{k}_K$-automorphisms. We obtain a chain of group homomorphisms 
\begin{equation}\label{eq:autom_inj} I_K/I_L\hookrightarrow \Aut(\Cc'_{\overline{k}_K})\to\Aut\!\left(H^1_{\et}(\Cc'_{\overline{k}_K},\Q_\ell)\right)\xrightarrow{\sim}\Aut\!\left(H^1_{\et}(C_{\overline{K}},\Q_\ell)\right).\end{equation}

\begin{prop}\label{prop:iso_Ca0}
Let $\s\in I_K^{\wild}$. The induced automorphism $\s_{\Cc'}$ on $\Cc'_{\overline{k}_K}$ descends to $k_L$, and $\Cc'_{k_L}$ is $k_L$-isomorphic to $C_{a,0}$ for some $a\in k_L^\times$. The automorphism of $C_{a,0}$ induced by $\s_{\Cc'}$ is given by $\s_a^r:(x,y)\mapsto (x+r,y)$ with some $r\in\F_p$. The image $\rho_\ell(\s)$ is nontrivial if and only if $r\neq0$.
\end{prop}

\begin{proof}
If $\rho_\ell(\s)=\Id$, then $\s\in I_L$, so $\s_{\Cc'}$ is the identity on $\Cc'_{\overline{k}_K}$ by \eqref{eq:autom_inj}. in the same way, if $\rho_\ell(\s)\neq\Id$, then the class of $\s$ in $I_K/I_L$ has order $5$, so it induces an automorphism on $\Cc'_{\overline{k}_K}$ of order $5$.

We have seen in Prop.~\ref{prop:almost_abelian} that the classes of $\s$ and $\varphi_L$ commute in $\Ga_K/I_L$. It follows that they commute as scheme-automorphisms of $\Cc'_{\overline{k}_L}$, which means that $\s_{\Cc'}$ descends to a $k_L$-automorphism of $\Cc'_{k_L}$.    

The main arguments for the second part are given in \autocite{roquette} and \autocite{homma_aut}, which we specialize to our situation. Let $\Ga\simeq C_5$ be the image of $I_K^{\wild}$ is $\Aut(\Cc'_{k_L})$, and let $\pi:\Cc'_{k_L}\to \Cc'_{k_L}/\Ga$ be the quotient map, which is defined over $k_L$. As a consequence of the Hurwitz formula, \autocite[Remark~1.2.(A).(b)]{homma_aut} shows that $\Ga$ fixes a unique closed point $P$ in $\Cc'_{k_L}$, and that $\Cc'_{k_L}/\Ga$ has genus zero. Since $\Ga$ commutes with $(\varphi_L)_{\Cc'}$, the point $(\varphi_L)_{\Cc'}(P)$ is also fixed by $\Ga$, so $(\varphi_L)_{\Cc'}(P)=P$, meaning that $P$ is a $k_L$-rational point. Then $\pi(P)$ is $k_L$-rational, so $\pi$ is in indeed a finite $k_L$-morphism $\Cc'_{k_L}\to \Pro^1_{k_L}$ of degree $5$ ramified only at $P$. Let $k_{L}(\Cc'_{k_L})$ denote the function field of $\Cc'_{k_L}$, then $k_L(\Cc'_{k_L})^{\Ga}$ is a rational function field over $k_L$, and we let $y$ be a generator having a (unique) pole at $P$. 

Since $k_L(\Cc'_{k_L})/k_L(y)$ is cyclic of order $5$, applying Artin--Schreier theory we have $k_L(\Cc'_{k_L})=k_L(x,y)$ satisfying an equation $x^5-x=f$ with $f\in k_L(y)$. Furthermore, since the pole of $y$ is the unique branch point of $\pi$, we may suppose that $f\in k_L[y]$. Since $\Cc'_{k_L}$ has genus 2, we must have $\deg f=2$. We may further suppose that $f(y)=ay^2+c$ with $a,c\in k_L$, $a\neq0$, thus we have a $k_L$-isomorphism $\Cc'_{k_L}\simeq C_{a,c}$.

With our particular choice of $L/K$ in \ref{subs:A_5_J_10}, the points of $J(C)[2]$ are rational over $L$. The isomorphism \eqref{eq:et_coh_fiber_iso} implies that the points of $J(C_{a,c})[2]$ are $k_L$-rational, which means that the polynomial $x^5-x-c$ splits completely over $k_L$. By translating $x$ with one of the roots we find that $\smash{\Cc'_{k_L}}\simeq C_{a,0}$  as $k_L$-schemes.

Lastly, every $\gamma\in\Ga$ fixes $y$, so $x-\gamma(x)$ is a root of $X^5-X=0$, thus giving $\gamma=\s_a^r$ for some $r\in\F_5$, and $r=0$ if and only if $\gamma$ is trivial.
\end{proof}

\begin{prop}\label{prop:repr_traces}
We fix $\s\in I_K^{\wild}$. Let $a\in k_L^\times$ and $r\in\F_5$ be as in Prop.~\ref{prop:iso_Ca0}. For every $m,n\in\Z$ we have \[\Tr\rho_\ell(\s^m\varphi_L^n)=\begin{cases}
-\lege{a}{k_L}^n 5^{\frac{n[k_L:\F_5]}{2}}& \text{ if $n[k_L:\F_5]$ is even and $5\nmid m$,} \\ 
\lege{a}{k_L}^n 4\cdot 5^{\frac{n[k_L:\F_5]}{2}}& \text{ if $n[k_L:\F_5]$ is even and $5\mid m$,}\\
-\lege{a}{k_L}^n \lege{rm}{\F_5}5^{\frac{n[k_L:\F_5]+1}{2}}& \text{ if $n[k_L:\F_5]$ is odd.}
\end{cases}\]
\end{prop}

\begin{proof}
By Prop.~\ref{prop:iso_Ca0}, the automorphism induced by $\s$ on $C_{a,0}$ is given by $\s_a^r:(x',y')\mapsto(x'+r, y')$. From the classical theory of the Frobenius actions on the étale cohomology group we know that the morphism $F_q:(x',y')\mapsto(x'^{q_L}, y'^{q_L})$ of $C_{a,0}$ induces the action of $\varphi_L$ on $H^1_{\et}(\Cc'_{\overline{k}_K},\Q_\ell)$. 

We fix a square root $\sqrt{a}\in\overline{k}_K$. Then there is a $\overline{k}_K$-isomorphism $C_{1,0}\to C_{a,0}$ given by $(x,y)\mapsto (x,\frac{y}{\sqrt{a}})$. Using this isomorphism we compute that the $\overline{k}_K$-automorphism on $C_{1,0}$ induced by $\s_a^r$ descends to $\F_5$ and is exactly $\s_1^r$. Similarly, $F_q$ induces $F^{[k_L:\F_5]}\circ\iota$ on $C_{1,0}$ if $\lege{a}{k_L}=-1$ or $F^{[k_L:\F_5]}$ if $\lege{a}{k_L}=1$.

Therefore, \[\Tr\rho_\ell(\s^m\varphi^n_L)=\lege{a}{k_L}^n\cdot\Tr\left(\s_1^{rm}\circ F^{n[k_L:\F_5]}\right)^*.\] The desired formulas follow from Lemma~\ref{lem:traces_C10}.
\end{proof}

\subsection{Square classes of differences of Weierstrass roots}\label{subs:diff_square} Let $Y^2=P(X)$ be a Weierstrass equation defining $C/K$ with $P\in K[X]$ unitary of degree $5$ as in Prop.~\ref{prop:hyperell_eq_spec}. In particular, $A_5=1$. Any element $\s\in I^{\wild}_K$ for which $\rho_\ell(\s)$ is nontrivial acts transitively on the roots of $P$. We fix a root $\a_1\in M$ of $P$, then the other roots are $\a_i:=\s^{i-1}(\a_1)$. Following Prop.~\ref{prop:iso_Ca0}, the exists $a\in k_L^\times$ such that $\Cc_{k_L}\simeq C_{a,0}$ over $k_L$, and $\s$ induces $\s_a^r\in\Aut(C_{a,0})$ for some $r\in\F_p^\times.$ We note that the curve $C_{a,0}$ is $k_L$-isomorphic to the curve defined by the equation $y'^2=x'^5-a^4x'$, where $y'=a^3y$ and $x'=ax$. We have $\s_a^r:(x',y')\mapsto(x'+ar,y')$. 

\begin{prop}\label{prop:diff_square} The following properties hold :
\begin{enumerate}
\item The valuation of $\a_i-\a_j$ is the same for every $i\neq j$;
\item Assume that $[k_L:\F_5]$ is odd. There exists $\s\in I_K^{\wild}$ such that $\lege{ar}{k_L}=1$. In this case, $\a_1-\s(\a_1)\in(L^\times)^2$. 
\end{enumerate}
\end{prop}

\begin{proof}
(1) Since $C_L/L$ has good reduction by Prop.~\ref{prop:wild5_good_equiv}, there exists an affine variable change over $L$ which transforms $Y^2=P(X)$ into a Weierstrass equation with coefficients in $\O_L$ and an invertible discriminant (see \autocite[Lemme~3]{liu_models}). An affine transformation modifies all $v_L(\a_i-\a_j)$ by adding the same constant $v$. The new discriminant has zero valuation so we must have $v_L(\a_i-\a_j)+v=0$ for all $i\neq j$.

(2) The existence of $\s$ such that $\lege{ar}{k_L}=1$ follows from $\lege{r}{k_L}=\lege{r}{\F_5}$.

The extension $H/K$ from \ref{subs:A_5_J_10} is at most tamely ramified, so $\s$ acts trivially on it. By Prop.\ref{prop:explicit_IM}, the extension $L/H$ is Galois, so $\s$ acts on $L$. Then, $\s$ induces an $L$-semilinear automorphism of $C_L/L$. 

The action of $\s$ on the function field $K(X,Y)$ with $Y^2=P(X)\in K[X]$ is trivial. Applying the change of variables $X=X'+\a_1$ gives the equation \[Y^2=P'(X'):=X'(X'-\a_2+\a_1)\ldots(X'- \a_5+\a_1)\in M[X'].\] The $\s$-action extends $M$-semilinearly to $M(X,Y)=M(X',Y)$ and \[\s(X')=X'-\a_2+\a_1.\]

With $P$ as in Prop.~\ref{prop:hyperell_eq_spec}, we have $40\mid e(L/K)v_K(\Delta(P))$ by Prop.~\ref{prop:inertia_disc_val}. Let $\varpi_L$ be any uniformizer of $L$ and $\delta:=\varpi_L^{\frac{e(L/K)v_K(\Delta(P))}{40}}$. After applying another change of variables $Y=\delta^5Y''$, $X'=\delta^2X''$ we obtain \[Y''^2=P''(X''):=X''\left(X''-\frac{\a_2-\a_1}{\delta^2}\right)\ldots\left(X''-\frac{\a_5-\a_1}{\delta^2}\right),\] and $\s(X'')=X''-\frac{\a_2-\a_1}{\delta^2}$. The formula \eqref{eq:var_change_disc} gives \[v_L\left(\Delta(P'')\right)=v_L\left(\delta^{-100} \cdot \delta^{60}\Delta(P)\right)=e(L/K)v_K\left(\Delta(P)\right)-40v_L(\delta)=0.\]

For all $i\neq j$, applying part (1) gives \[v_L\left(\frac{\a_i-\a_j}{\delta^2}\right)=\frac{1}{20}v_L(\Delta(P))-2v_L(\delta)=0.\] 

It follows that $Y''^2=P''(X')$ defines a smooth model $\mathcal{W}/\O_L$ of $C_L/L$, which is unique up to isomorphism. Its reduction $\Wc_{k_L}/k_L$ must be $k_L$\mydash{}isomorphic to the curve $C_{a,0}/k_L$, defined by $y'^2=x'^5-a^4x'$. Let $x''$ denote the class of $X''$ in the function field of $\Wc_{k_L}$. By construction, the points at infinity of both of these models are fixed by the $k_L$-linear automorphisms induced by $\s$. Since on each curve there is a unique such fixed point (proven in \autocite{homma_aut}), there must be an affine variable change $y''=ay'$ $x''=bx'+c$ for some $a,b,c \in k_L$. Then $b^5=a^2$, so $b$ is a square in $k_L$. 

On one hand, as pointed out in \ref{subs:diff_square}, we have \[\s(x'')=b\s(x')+c=bx'+bar+c,\] and on the other hand, from the construction of $P''$, we have \[\s(x'')=bx'+c+\left(\frac{\a_1-\a_2}{\delta^2}\bmod \m_L\right).\] Thus, the class of $\frac{\a_1-\a_2}{\delta^2}$ in $k_L$ is $bar$, which is a square, so $\a_1-\a_2\in (L^\times)^2$.
\end{proof}

\begin{prop}\label{prop:root_diff_square}\vphantom{a}
\begin{enumerate}
\item We have $H\subset M$. 
\item For all $k\neq l$, the element $\a_k-\a_l$ is a square in $L(\zeta_8)$.
\end{enumerate}
\end{prop}

\begin{proof}
For (2), by replacing $\s$ with some power, without loss of generality we may assume that $k=1$, $l=2$. Applying \eqref{eq:disc_unitary} gives 
\begin{equation}\label{eq:beta_expr} -\beta=\Delta(P)=2^{8}\prod_{i<j}(\a_i-\a_j)^2=2^{8}(\a_1-\a_2)^{20}\prod_{i<j}\left(\frac{\a_i-\a_j}{\a_1-\a_2}\right)^2.\end{equation} 
The wild ramification group $I^{\wild}(M/K)$ acts trivially on $M^\times/U^1_M$, so \[\frac{\a_i-\a_1}{\a_2-\a_1}=\sum_{k=0}^{i-2}\frac{\s^k(\a_2-\a_1)}{\a_2-\a_1}\equiv i-1\bmod \m_M.\] Then \[\prod_{i<j}\left(\frac{\a_i-\a_j}{\a_1-\a_2}\right)^2\equiv \prod_{i<j}(j-i)^2\equiv(288)^2\equiv -1 \bmod \m_M.\]
Since $U_M^1$ is $8$-divisible, it follows that $\beta\in (M^\times)^4$, thus giving (1). Recall that $\b\in(L^\times)^8$, thus $(\a_1-\a_2)^4\in(L^\times)^8$. It follows that $\a_1-\a_2$ is a square in $L(\zeta_8)^\times$, thus proving (2).
\end{proof}

\begin{rem}
Prop.~\ref{prop:root_diff_square} must be contrasted with Prop.~\ref{prop:diff_square}. Unless $L=L(\zeta_8)$, only half of the differences $\a_i-\a_j$ are squares in $L$. 
\end{rem}

\begin{prop}\label{prop:a-vdelta_even}
For any discriminant $\Delta$ of $C/K$ we have $v_K(\Delta)\equiv a(\rho_\ell)\!\bmod \!2$. 
\end{prop}

\begin{proof}
This is derived from \autocite{liu_g2_disc_cond}. First, $v_K(\Delta)\equiv v_K(\Delta_{\min})\bmod 2$ for $\Delta_{\min}$ associated to a so-called minimal equation. Then, using \autocite[Prop.~1, Thm.~1]{liu_g2_disc_cond} we have $v_K(\Delta_{\min})-a(\rho_\ell)=m-1+\frac{d-1}{2}$ where $m$ is the number of irreducible components of the special geometric fiber of the minimal regular model of $C/K$ and $d$ is a geometric invariant of $C/K$ defined in \autocite[\S5.2]{liu_g2_disc_cond}. Finally, for each possible geometric type of the minimal regular model Liu computes $d$. In our case, $m$ is $1,3,4,5,9,11,$ or $13$ and $d=1$ for each value of $m$, except $d=3$ when $m=4$. 
\end{proof}

\section{Maximal inertia action over $5$-adic fields}\label{sect:max_inertia}

We continue in the setting of \ref{subs:5-adic_wild}.

\begin{prop}\label{prop:hyperell_max_equiv}
The following are equivalent :
\begin{enumerate}
\item $v_K(\Delta)$ is odd for any discriminant $\Delta$ of $C/K$;
\item The extension $M/K$ is totally ramified and $\Gal(M/K)\simeq F_5$;
\item $\rho_\ell(I_K)\simeq C_5\rtimes C_8$;
\item $a(C/K)$ is odd.
\end{enumerate}
\end{prop}

\begin{proof}
Prop.~\ref{prop:inertia_disc_val} gives (1)$\Leftrightarrow$(3), and Prop.~\ref{prop:a-vdelta_even} gives (1)$\Leftrightarrow$(4). Prop.~\ref{prop:explicit_IM} shows that $\rho_\ell(I_K)$ has a quotient isomorphic to the inertia subgroup of $\Gal(M/K)$. Then \ref{subs:inertia_class} shows that (2) implies (3). Suppose (3), then $L/K$ has ramification index $40$. By \ref{subs:red_torsion}, the ramification index of $M/K$ is at least $20$. Statement~(2) now follows from Prop.~\ref{prop:hyperell_irred_5}. 
\end{proof} 

\subsection{Maximal ramification hypothesis}\label{subs:max_inertia_hyp} From now on we suppose that $\rho_\ell(I_K)\simeq C_5\rtimes C_8$. The complex Weil--Deligne representation attached to $\rho_\ell$ is given by the Weil representation $\rho:=\rho_\ell|_{W_K}\otimes_{\Q_\ell}\C$ and the trivial monodromy operator.

\begin{prop}\label{prop:L_tot_ram}
The extension $L/K$ is totally ramified, and $[L:M]=2$, $[M:H]=5$, and $[H:K]=4$.
\end{prop}

\begin{proof}
It follows from Prop.~\ref{prop:root_diff_square}.(1) that $H\subset N\cap M$. Prop.~\ref{prop:hyperell_max_equiv} shows that $M/K$ is totally ramified of degree $20$ and that $[H:K]=4$. Therefore, $M/H$ is totally ramified of degree $5$, and $N\cap M = H$. From Prop.~\ref{prop:hyperell_max_equiv} we also see that $[N:H]=2$ and that $L/K$ has ramification index $40$, so $[L:K]=[H:K][N:H][M:H]=40$. It follows that $L/K$ is totally ramified. 
\end{proof}

\begin{prop}\label{prop:tot_gal_max}
In the notation of \autocite{group_names}, we have \[\Gal(L(\zeta_8)/K)\simeq\begin{cases}C_5\rtimes C_8 & \text{if $[k_K:\F_5]$ is even,} \\ C^2_2.F_5 &\text{if $[k_K:\F_5]$ is odd.}\end{cases}\]
\end{prop}

\begin{proof}
The inertia subgroup $I(L(\zeta_8)/K)\subset \Gal(L(\zeta_8)/K)$ is isomorphic to $C_5\rtimes C_8$ and has index at most $2$ (from Prop.~\ref{prop:explicit_IM} and Prop.~\ref{prop:L_tot_ram}). 

It remains to show that if $L(\zeta_8)/L$ is nontrivial, then $\Gal(L(\zeta_8)/K)\simeq C^2_2.F_5$. In this case we have $\Gal(L(\zeta_8)/M)\simeq C_2^2$ since $L/M$ is totally ramified of degree $2$. It follows from Prop.~\ref{prop:hyperell_max_equiv} that $\Gal(L(\zeta_8)/K)$ is an extension $G$ of $F_5$ by $C_2^2$. The extension cannot be split, because otherwise $\Gal(L(\zeta_8)/K)$ would have $C_2^2\rtimes C_4$ as a $2$-Sylow subgroup, which has exponent $4$ and therefore has no element of order $8$. In order to identify $\Gal(L(\zeta_8)/K)$ as $C^2_2.F_5$ by using \autocite{group_names} we are left to show that the extension $G$ is non-central, i.e. that the subgroup $C_2^2\subset G$ which identifies with $\Gal(L(\zeta_8)/M)\subset \Gal(L(\zeta_8)/K)$ is non-central. Indeed, $\Gal(L(\zeta_8)/M)$ cannot be central because $\Gal(L/K)$ is non-Galois. \end{proof}

\begin{prop}\label{prop:40_induced} Under the hypothesis of \ref{subs:max_inertia_hyp} the following statements hold :
\begin{enumerate}
\item The representation $\rho$ is irreducible;
\item There exists characters $\chi$ and $\chi'$ of $W_H$ such that 
\begin{equation}\label{eq:restr_fact}
\rho|_{W_H}\simeq\chi\oplus \chi^{-1}(-1)\oplus \chi'\oplus\chi'^{-1}(-1);
\end{equation}
\item If $\chi$ is any of the four direct summands in \eqref{eq:restr_fact}, then \[\rho\simeq \Ind_{W_H}^{W_K}\chi,\] and the Artin conductor $a(\chi)$ is even.
\end{enumerate}
\end{prop} 

\begin{proof}
We observe that every irreducible representation of $C_5\rtimes C_8$ necessarily has dimension 1 or 4 (see, e.g., \autocite{group_names}). It follows that $\rho|_{I_K}$ is irreducible since it cannot be a direct sum of $1$\mydash{}dimentional representations. Thus, (1) holds. 

The extension $L/H$ is the compositum of the $C_5$-extension $M/H$ and the quadratic extension $N/H$, so $\Gal(L/H)\simeq C_{10}$. It follows that $LK^{\unr}/H$ is abelian. Therefore, $\rho|_{W_H}$ has abelian image and splits into $1$-dimensional factors
\begin{equation}\label{eq:decomp_chi} \rho|_{W_H}\simeq\chi_1\oplus\chi_2\oplus\chi_3\oplus\chi_4.\end{equation}
Frobenius reciprocity gives a nontrivial morphism of representations \[\Ind_{W_H}^{W_K}\chi_1\to \rho.\] Since $\rho$ is irreducible, the morphism is surjective and, in fact, is an isomorphism because $\dim\Ind_{W_H}^{W_K}\chi_1 =4 = \dim\rho$. Thus, (3) holds. 

Since $\rho$ is wildly ramified, by using an explicit construction of the induced representation we observe that $\chi_1$ must be wildly ramified. 

The twisted representation $\rho(\frac{1}{2})$ is symplectic with respect to the Weil pairing, so, in particular, the dual of $\rho$ is $\rho^*\cong \rho(1)$ and $\det\rho=\chi_{\unr}^{-2}$. Then \eqref{eq:decomp_chi} gives \[\rho|_{W_H}\simeq \left(\rho|_{W_H}\right)^*(-1)\simeq\chi_1^{-1}(-1) \oplus \chi_2^{-1}(-1)\oplus \chi_3^{-1}(-1)\oplus \chi_4^{-1}(-1).\] The wild ramification of $\chi_1$ implies that $\chi_1\not\simeq \chi_1^{-1}(-1)$, so we may suppose that $\chi_2\simeq\chi_1^{-1}(-1)$. We then have $\chi_4\simeq\chi_3^{-1}(-1)$. Posing $\chi=\chi_1$ and $\chi'=\chi_3$ gives (2).

By Prop.~\ref{prop:hyperell_max_equiv}, $a(C/K)$ is odd. Since $H/K$ is totally tamely ramified of degree $4$, from \autocite[\S10.(a2)]{rohrlich} we have $a(C/K)=a(\rho)=a(\chi)+3$.\end{proof}

\begin{prop}\label{prop:rat_4-tors}
Each point of $J(C)[4]$ is rational over $L(\zeta_8)$.
\end{prop}

\begin{proof}
Let $Y^2=P(X)$ be as in Prop.~\ref{prop:hyperell_eq_spec}, and let $\a_1,\ldots,\a_5\in\overline{K}$ be the roots of $P$, so that $M=K(\a_1,\ldots,\a_5)$. Let $\widetilde{M}:=\Q(\sqrt{-1},\a_1,\ldots,\a_5)\subset M$. Then, $C$ and $J(C)$ are defined over $\widetilde{M}$, and it follows from \autocite[Remark~4.2]{yelton_4-tors} that 
\[\widetilde{M}(J(C)[4])=\widetilde{M}\left(\left(\sqrt{\a_i-\a_j}\right)_{i<j}\right).\] The proposition now follows from Prop.~\ref{prop:root_diff_square}.(2).
\end{proof}

\begin{cor}\label{cor:twisted_trivial}
The map $\rho(\varphi_{L(\zeta_8)})$ is given as multiplication by the scalar $ \sqrt{q_{L(\zeta_8)}}$. As an immediate consequence, the twisted representation $\rho(\frac{1}{2})$ is trivial on $W_{L(\zeta_8)}$.
\end{cor}

\begin{proof} 
Since $\rho(\varphi_{L(\zeta_8)})$ is central in $\Img(\rho)$ by Prop.~\ref{prop:almost_abelian}, it acts as multiplication by a scalar $z\in\C^\times$. From \eqref{eq:restr_fact} we see that $z=z^{-1}q_{L(\zeta_8)}$, so $z=\pm\sqrt{q_{L(\zeta_8)}}$. We note that $\sqrt{q_{L(\zeta_8)}}$ is always an integral power of $5$, thus, in particular, $z\equiv\pm1\bmod4$. On the other hand, Prop.~\ref{prop:rat_4-tors} implies that $\rho_2(\varphi_{L(\zeta_8)})\in\Aut_{\Z_2}\left(H^1_{\et}(C_{\overline{K}},\Z_2)\right)$ satisfies $\rho_2(\varphi_{L(\zeta_8)})\equiv \Id{}\bmod4$. We therefore conclude that $z=\sqrt{q_{L(\zeta_8)}}$.
\end{proof}

\section{Computation of root numbers}\label{sect:max_proof}

We assume the hypotheses of \ref{subs:5-adic_wild} and \ref{subs:max_inertia_hyp} and prove our main result.

\begin{thm}\label{thm:max_ramif_rootN}
Let $a_6$ be as in Prop.~\ref{prop:hyperell_eq_spec}, and let $\Delta$ be the discriminant associated to any Weierstrass equation defining $C/K$. The root number of $C/K$ is given by
\[w(C/K)=(-1)^{[k_K:\F_5]+1}\cdot\lege{v_K(a_6)}{k_K}\cdot (\Delta,a_6)_K.\]
\end{thm}

Let $\psi_k:K\to\C^\times$ be the additive character from \ref{subs:add_char_ex}. For the general theory and the formulas of root numbers the reader may refer to \autocite{rohrlich}.

\subsection{Root number of an induced representation}\label{subs:root_induced} We have $\rho=\Ind_{W_H}^{W_K}\chi$ from Prop.~\ref{prop:40_induced}, so the formula \autocite[\S11.($\epsilon$2)]{rohrlich} gives 
\begin{equation}\label{eq:root_n_ind_formula} w(C/K)=w(\chi,\psi_k\circ \Tr_{K/H})\cdot w(\Ind_{W_H}^{W_K}\1,\psi_k). \end{equation}

\begin{lem}\label{lem:root_n_ind_id_4}
We have $w(\Ind_{W_H}^{W_K}\1,\psi_k)=-1$.
\end{lem}

\begin{proof} The representation $\Ind_{W_H}^{W_K}\1$ is isomorphic to the regular representation of $\Gal(H/K)\simeq C_4$. Let $\chi_4:W_K\to \C^\times$ denote a totally ramified character of order $4$ such that $\ker\chi_4=W_{H}$. We then have a decomposition \begin{equation}\label{eq:ind_1_dec} \Ind_{W_H}^{W_K}\1\simeq \1\oplus \chi_4^2\oplus\chi_4\oplus \chi_4^{-1},\end{equation} and thus multiplicativity of root numbers \autocite[\S11.($\epsilon$1)]{rohrlich} gives 
\[w(\Ind_{W_H}^{W_K}\1,\psi_k)=w(\chi_4^2,\psi_k)\cdot w(\chi_4\oplus \chi_4^{-1},\psi_k).\]
The properties from \autocite[\S12~Lemma]{rohrlich} give \[w(\chi_4\oplus \chi_4^{-1},\psi_k)=\chi_4(\theta_K(-1)),\] where $\theta_K$ is Artin's reciprocity map. We have $\chi_4(\theta_K(-1))=1$ exactly when $-1$ is a 4th power in $K^\times$, so \[w(\chi_4\oplus \chi_4^{-1},\psi_k)=(-1)^{[k_K:\F_5]}.\] In order to compute $w(\chi_4^2,\psi_k)$ we apply the formula \autocite[\nopp (8.7.1)]{abbes_saito} with $\beta=1$ there and $\tau(\chi_4^2,\psi_k)=-G_{[k_K:\F_5]}(\chi_4^2)=(-\sqrt{p})^{[k_K:\F_5]}$ (we use \autocite[Thm.~11.5.2]{gauss_sums_book}), which gives \[w(\chi_4^{2},\psi_k)=(-1)^{[k_K:\F_5]+1},\] thus completing the proof of the lemma. 
\end{proof}

\subsection{Connection with a Weierstrass equation}\label{subs:gauge_weierstrass_eq} Let $Y^2=P(X)$ be the Weierstrass equation for $C/K$ from Prop.~\ref{prop:hyperell_eq_spec}. We fix a root $\a_1\in M$ of $P$, then $M=H(\a_1)$. Let $\chi$ be as in Prop.~\ref{prop:40_induced}, and let $\s\in I_H$ be an element such that 
\begin{equation}\label{eq:chi_sigma_img}\chi(\s)=e^{\frac{2\pi i}{5}}.\end{equation}
It follows that $\s$ restricts to a generator of $\Gal(M/H)\simeq C_5$. Let $\a_j:=\s^{j-1}(\a_1)$ be the roots of $P$, and let $d_{\a_1,\chi}:=\norm_{M/H}(\a_1-\a_2).$ We have $\norm_{M/H}(\a_1)=-a_6$ and \begin{equation}v_M(\a_1)=v_H(\norm_{M/H}(\a_1))=v_H(a_6)=4v_K(a_6).\end{equation} Since $a(\chi)$ is even by Prop.~\ref{prop:40_induced}.(3), applying Cor.~\ref{cor:char_rootN_D} (with $K=H$ there) gives (recall the notation $\approx$ from \ref{subs:root_n_char})
\begin{align}
w(\chi,\psi_k\circ\Tr_{H/K})& \approx \chi\circ\theta_H\left(v_M(\a_1)\cdot \norm_{M/H}(\a_1)\right)^{-1}\cdot \chi\circ\theta_H(d_{\a_1,\chi}) \nonumber \\
& \approx \chi\circ\theta_H(-4v_K(a_6)a_6)^{-1}\cdot \chi\circ\theta_H(d_{\a_1,\chi}). \label{eq:chi_rootN_inter}
\end{align}

Recall that $\det\rho=\chi_{\unr}^{-2}$. Let $t:W_K^{\ab}\to W_H^{\ab}$ be the transfer map. Deligne's determinant formula \autocite[508]{deligne_eq_fonctionelle} gives 
\[ \chi_{\unr}^{-2}=\det\Ind_{W_H}^{W_K}\chi=\det\Ind_{W_H}^{W_K}\1 \cdot \chi\circ t. \] Composing with $\theta_K$ and taking into account the decomposition \eqref{eq:ind_1_dec} gives 
\[||\cdot||^{-2}_K=\chi_4^2\circ\theta_K \cdot (\chi\circ\theta_H)|_{K^\times}.\]
Since $-4v_K(a_6)a_6\in K^\times$ and $||\cdot||_K\approx 1$, the above gives 
\begin{equation}\label{eq:chi_on_K} \chi\circ\theta_H(-4v_K(a_6)a_6) \approx \chi_4^2\circ\theta_K(-4v_K(a_6)a_6).\end{equation}

Since $-\b$ is a norm from $K(\sqrt{\b})$, we have  $\chi_4^2\circ\theta_K(-\b)=1$. Therefore, $\chi_4^2\circ\theta_K$ is equal to the Hilbert symbol $(\beta,\cdot)_K$, since both are quadratic ramified characters trivial on $-\beta$. Since $\beta$ differs from any discriminant $\Delta$ of $C/K$ by a square in $K^\times$, we have $(\beta,\cdot)_K=(\Delta,\cdot)_K$. Applying this to \eqref{eq:chi_on_K} together with the formula \autocite[V.(3.4)]{neukirch} gives 
\begin{equation}\label{eq:chi_rootN_a6} \chi\circ\theta_H(-4v_K(a_6)a_6)\approx \lege{v_K(a_6)}{k_K}\cdot(\Delta,a_6)_K.\end{equation}
Plugging \eqref{eq:chi_rootN_a6} into \eqref{eq:chi_rootN_inter} we obtain
\begin{equation}\label{eq:chi_rootN_a6_m} w(\chi,\psi_k\circ \Tr_{H/K})\approx\lege{v_K(a_6)}{k_K}\cdot(\Delta,a_6)_K\cdot \chi\circ\theta_H(d_{\a_1,\chi}).\end{equation}

\begin{lem}\label{lem:chi_d_even}
If $[k_K:\F_5]$ is even, then $\chi\circ\theta_H(d_{\a_1,\chi})\approx 1.$
\end{lem}

\begin{proof}
Here we have $L(\zeta_8)=L$. Then Prop.~\ref{prop:root_diff_square}.(2) implies that \[\a_1-\a_2\in\norm_{L(\zeta_8)/M}(L(\zeta_8)^\times),\] thus $d_{\a_1,\chi}$ is a norm from $L(\zeta_8)^\times$. Since $\rho(\frac{1}{2})=\Ind_{W_H}^{W_K}(\chi(\frac{1}{2}))$, Cor.~\ref{cor:twisted_trivial} implies that $\chi(\frac{1}{2})\circ\theta_H$ is trivial on $\norm_{L(\zeta_8)/H}(L(\zeta_8)^\times)$. Then,
\begin{equation} \chi\circ\theta_H(d_{\a_1,\chi})=||d_{\a_1,\chi}||_H^{-\frac{1}{2}} \cdot \left(\chi({\textstyle\frac{1}{2}})\circ\theta_H\right)(d_{\a_1,\chi})=||d_{\a_1,\chi}||_H^{-\frac{1}{2}}\approx 1. \qedhere \end{equation}
\end{proof}

\begin{lem}\label{lem:chi_gauss_frob}
Suppose that $[k_K:\F_5]$ is odd. Let $a\in k_L$ and $r\in\F_5$ be associated to $\s$ as in Prop.~\ref{prop:iso_Ca0}. Then for every geometric Frobenius lift $\varphi_L\in W_L$, we have \[\chi(\varphi_L)=-\lege{ar}{k_L}\sqrt{q_K}.\]
\end{lem}

\begin{proof}
Recall from Prop.~\ref{prop:L_tot_ram} that $L/K$ is totally ramified, so $k_L=k_K$. Since $[k_K:\F_5]$ is odd, $q_K=q_L=\sqrt{q_{L(\zeta_8)}}$, and $\lege{\cdot}{\F_5}$ is the restriction of $\lege{\cdot}{k_L}$ to $\F_5$.
 
Let $\chi'$ be the other character appearing in Prop.~\ref{prop:40_induced}. From Prop.~\ref{prop:repr_traces} we have $\Tr\rho(\s)=-1$, which, together with \eqref{eq:chi_sigma_img}, forces \begin{equation}\label{eq:chis_sigma}
\chi'(\s)\in\left\{(e^{\frac{2\pi i}{5}})^{2}, (e^{\frac{2\pi i}{5}})^{3}\right\}.\end{equation}

Cor.~\ref{cor:twisted_trivial} implies that the eigenvalues of $\rho(\varphi_L)$ are $\pm \sqrt{q_L}$. From Prop.~\ref{prop:repr_traces} we have $\Tr\rho(\varphi_L)=0$, so there exists some $w=\pm 1$ such that \begin{equation}\label{eq:chis_frob}\chi(\varphi_L)=w\sqrt{q_L} \hspace{0.5cm}\text{and}\hspace{0.5cm} \chi'(\varphi_L)=-w\sqrt{q_L}.\end{equation}

Using \eqref{eq:chi_sigma_img}, \eqref{eq:chis_sigma}, and \eqref{eq:chis_frob} together with a classical formula for Gauss sums \autocite[\S1.1]{gauss_sums_book} gives 
\[\Tr\rho(\s\varphi_L)=w\sqrt{q_L}\left(e^{\frac{2\pi i}{5}}+(e^{\frac{2\pi i}{5}})^{4}-(e^{\frac{2\pi i}{5}})^{2}-(e^{\frac{2\pi i}{5}})^{3}\right)=w\sqrt{5q_L}.\]
It now follows from Prop.~\ref{prop:repr_traces} that $w=-\lege{ar}{k_L}$.
\end{proof}

\subsection{Choosing $\chi$}\label{subs:choice_chi} We assume that $[k_K:\F_5]$ is odd, then $[k_L:\F_5]$ is also odd. Although the root number $w(\chi,\psi_k\circ\Tr_{H/K})$ does not depend on the choice of the character $\chi$ in Prop.~\ref{prop:40_induced}.(3), in order to carry out a detailed computation we will need to fix a particular $\chi$. Depending on whether $a$ is a square in $k_L^\times$, we may choose $\s$ and, consequently, $\chi$ so that $\lege{ar}{k_L}=1$ and that we still have \eqref{eq:chi_sigma_img}. 

\begin{lem}\label{lem:chi_d_odd}
If $[k_K:\F_5]$ is odd and $\chi$ is as in \ref{subs:choice_chi}, then $\chi\circ\theta_H(d_{\a_1,\chi})\approx -1.$
\end{lem}

\begin{proof} Applying Lemma~\ref{lem:chi_gauss_frob} for the chosen $\chi$ gives 
\[\chi(\varphi_L)=-\sqrt{q_K}.\] On the other hand, Prop.~\ref{prop:diff_square}.(2) tells us that $\a_1-\a_2$ is a square in $L$, so there exists some $b\in L$ such that $\a_1-\a_2=\norm_{L/M}(b)$. Prop.~\ref{prop:diff_square}.(1) and Prop.~\ref{prop:L_tot_ram} give \[v_L(b)=v_M(\a_1-\a_2)=\frac{1}{20}v_M(\Delta(P))= v_K(\Delta(P)).\] It now follows from Prop.~\ref{prop:hyperell_max_equiv} that $v_L(b)$ is odd.
The restriction $\chi|_{W_L}$ is unramified, so the above discussion shows that \[\chi\circ\theta_H(d_{\a_1,\chi})=\chi\circ\theta_L(b)=\chi(\varphi_L)^{v_L(b)}=\left(-\sqrt{q_K}\right)^{v_L(b)}\approx -1.\qedhere\]
\end{proof}

\begin{proof}[Proof of Theorem~\ref{thm:max_ramif_rootN}] When $[k_K:\F_5]$ is even we use Lemma~\ref{lem:chi_d_even}, and when $[k_K:\F_5]$ is odd we choose $\chi$ as in \ref{subs:choice_chi} and use Lemma~\ref{lem:chi_d_odd} in order to obtain $\chi\circ\theta_H(d_{\a_1,\chi})\approx(-1)^{[k_K:\F_5]}$. Plugging the latter into \eqref{eq:chi_rootN_a6_m} gives 
\begin{equation}\label{eq:chi_rootN_final} w(\chi,\psi_k\circ\Tr_{H/K})\approx(-1)^{[k_K:\F_5]}\cdot\lege{v_K(a_6)}{k_K}\cdot (\Delta,a_6)_K.\end{equation} 
Combining \eqref{eq:chi_rootN_final} and Lemma~\ref{lem:root_n_ind_id_4} into \eqref{eq:root_n_ind_formula} proves the relation $\approx$ between the two sides of the formula of Thm.~\ref{thm:max_ramif_rootN}. Since both sides take values in $\{1,-1\}$, the theorem holds (see \ref{subs:root_n_char}). 
\end{proof}

\subsection{An example} This is \autocite[\href{https://www.lmfdb.org/Genus2Curve/Q/896875/a/896875/1}{genus $2$ curve 896875.a.896875.1}]{lmfdb}. Let $C/\Q$ be the hyperelliptic curve defined by \[Y^2=P(X):=X^5 + \frac{5}{4}X^4 - \frac{5}{2}X^3 - \frac{5}{4}X^2 + \frac{5}{2}X + \frac{1}{4}.\] Its discriminant is $\Delta=-5^5 \cdot 7 \cdot 41$, and a smooth model over $\Z_2$ can be given. 

Over $7$ the reduction is semi-stable, the singular point of the special fiber is given by $X=5$, $Y=0$, and we have $P(X)\equiv(X-5)^2H_7(X)\bmod 7$ with $H_7(X)=X^3 +6X^2 + X + 4$ separable over $\F_7$. Over $41$ the reduction is again semi-simple, the singular point is at $X=12$, $Y=0$, and we have $P(X)\equiv (X-12)^2H_{41}(X)\bmod 41$ with $H_{41}(X)=X^3+15X^2+29X+21$ separable over $\F_{41}$. We apply \autocite[Lemma~6.7]{brumer_kramer_sabitova} to compute $w(C/\Q_7)=-\lege{H_7(5)}{\F_7}=-\lege{4}{\F_7}=-1$ and $w(C/\Q_{41})=-\lege{H_{41}(12)}{\F_{41}}=-\lege{34}{\F_{41}}=1.$

We observe that $P(X+1)$ is Eisenstein over $\Z_5$, so the $\Ga_{\Q_5}$-action on $J(C)[2]$ is wildly ramified. Thus, $\rho_\ell$ is wildly ramified for every $\ell\neq5$, and $C/\Q_5$ has potentially good reduction by Prop.~\ref{prop:wild5_good_equiv}. The equation $Y^2=P(X+1)$ satisfies the conditions of Prop.~\ref{prop:hyperell_eq_spec} with $a_6=\frac{5}{4}$, so Thm.~\ref{thm:max_ramif_rootN} applies to give \[w(C/\Q_5)=\left(-5^5 \cdot 7 \cdot 41,\frac{5}{4}\right)_{\Q_5}=-1.\] 

The global root number is then $w(C/\Q)=1$, which is compatible with the Hasse--Weil and the BSD conjectures since both analytic and Mordeil--Weil ranks of $J(C)/K$ are $2$ (see \autocite{lmfdb}).

\printbibliography

@article {abbes_saito,
    AUTHOR = {Abbes, Ahmed and Saito, Takeshi},
     TITLE = {Local {F}ourier transform and epsilon factors},
   JOURNAL = {Compos. Math.},
  FJOURNAL = {Compositio Mathematica},
    VOLUME = {146},
      YEAR = {2010},
    NUMBER = {6},
     PAGES = {1507--1551},
   MRCLASS = {14F20 (11S15)},
  MRNUMBER = {2735372},
MRREVIEWER = {Lei Fu},
       DOI = {10.1112/S0010437X09004631},
}

@article {bisatt,
    AUTHOR = {Bisatt, Matthew},
     TITLE = {Explicit root numbers of abelian varieties},
   JOURNAL = {Trans. Amer. Math. Soc.},
  FJOURNAL = {Transactions of the American Mathematical Society},
    VOLUME = {372},
      YEAR = {2019},
    NUMBER = {11},
     PAGES = {7889--7920},
      ISSN = {0002-9947},
   MRCLASS = {11G40 (11G10)},
  MRNUMBER = {4029685},
MRREVIEWER = {Timo Keller},
       DOI = {10.1090/tran/7926},
       URL = {https://doi-org.scd-rproxy.u-strasbg.fr/10.1090/tran/7926},
}

@misc{bisatt_wild,
      title={Root number of the Jacobian of $y^2=x^p+a$}, 
      author={Matthew Bisatt},
      year={2021},
      eprint={2102.05720},
      archivePrefix={arXiv},
      primaryClass={math.NT}
}

@book {BLRNeron,
    AUTHOR = {Bosch, Siegfried and L\"utkebohmert, Werner and Raynaud, Michel},
     TITLE = {N\'eron models},
    SERIES = {Ergeb. Math. Grenzgeb. (3)},
    VOLUME = {21},
 PUBLISHER = {Springer-Verlag, Berlin},
      YEAR = {1990},
     PAGES = {x+325},
   MRCLASS = {14K15 (11G10 14L15)},
  MRNUMBER = {1045822},
MRREVIEWER = {James Milne},
       DOI = {10.1007/978-3-642-51438-8},
}

@article {brumer_kramer_sabitova,
    AUTHOR = {Brumer, Armand and Kramer, Kenneth and Sabitova, Maria},
     TITLE = {Explicit determination of root numbers of abelian varieties},
   JOURNAL = {Trans. Amer. Math. Soc.},
  FJOURNAL = {Transactions of the American Mathematical Society},
    VOLUME = {370},
      YEAR = {2018},
    NUMBER = {4},
     PAGES = {2589--2604},
      ISSN = {0002-9947},
   MRCLASS = {11G10 (11G25 11S20)},
  MRNUMBER = {3748578},
MRREVIEWER = {Gabriele Ranieri},
       DOI = {10.1090/tran/7116},
       URL = {https://doi-org.scd-rproxy.u-strasbg.fr/10.1090/tran/7116},
}

@article {chai_semiab,
    AUTHOR = {Chai, Ching-Li},
     TITLE = {N\'eron models for semiabelian varieties: congruence and change
              of base field},
      NOTE = {Loo-Keng Hua: a great mathematician of the twentieth century},
   JOURNAL = {Asian J. Math.},
  FJOURNAL = {Asian Journal of Mathematics},
    VOLUME = {4},
      YEAR = {2000},
    NUMBER = {4},
     PAGES = {715--736},
%      ISSN = {1093-6106},
       DOI = {10.4310/AJM.2000.v4.n4.a1},
}

@misc{coppola_max,
    title={Wild Galois representations: a family of hyperelliptic curves with large inertia image},
    author={Nirvana Coppola},
    year={2020},
    eprint={2001.08287},
    archivePrefix={arXiv},
    primaryClass={math.NT}
}

@incollection {deligne_eq_fonctionelle,
    AUTHOR = {Deligne, Pierre},
     TITLE = {Les constantes des \'equations fonctionnelles des fonctions
              {$L$}},
 BOOKTITLE = {Modular functions of one variable, {II} ({P}roc. {I}nternat.
              {S}ummer {S}chool, {U}niv. {A}ntwerp, {A}ntwerp, 1972)},
     PAGES = {501--597},
    SERIES = {Lecture Notes in Math.}, 
    VOLUME = {349},
 PUBLISHER = {Springer, Berlin},
      YEAR = {1973},
   MRCLASS = {12A70 (10H10 12A65 12B25 12B30)},
  MRNUMBER = {0349635},
MRREVIEWER = {Roger E. Howe},
}

@article {dd_root_ellc2,
    AUTHOR = {Dokchitser, Tim and Dokchitser, Vladimir},
     TITLE = {Root numbers of elliptic curves in residue characteristic 2},
   JOURNAL = {Bull. Lond. Math. Soc.},
  FJOURNAL = {Bulletin of the London Mathematical Society},
    VOLUME = {40},
      YEAR = {2008},
    NUMBER = {3},
     PAGES = {516--524},
      ISSN = {0024-6093},
   MRCLASS = {11G05 (11F80 11G07 11G40)},
  MRNUMBER = {2418807},
MRREVIEWER = {Anupam Saikia},
       DOI = {10.1112/blms/bdn034},
       URL = {https://doi-org.scd-rproxy.u-strasbg.fr/10.1112/blms/bdn034},
}

@book {gauss_sums_book,
    AUTHOR = {Berndt, Bruce C. and Evans, Ronald J. and Williams, Kenneth
              S.},
     TITLE = {Gauss and {J}acobi sums},
    SERIES = {Canadian Mathematical Society Series of Monographs and
              Advanced Texts},
      NOTE = {A Wiley-Interscience Publication},
 PUBLISHER = {John Wiley \& Sons, Inc., New York},
      YEAR = {1998},
     PAGES = {xii+583},
      ISBN = {0-471-12807-4},
   MRCLASS = {11L05 (11A15 11L10 11T22 11T24)},
  MRNUMBER = {1625181},
MRREVIEWER = {Charles Helou},
}

@online{group_names,
 author = {Dokchitser, Tim},
 title = {\textit{{G}roup{N}ames}},
 year = {2020},
 month=2,
 url = {http://groupnames.org},
 shorthand = {GN},
 %urldate = {2020-02-16},
}

@article {hesse_AS,
    AUTHOR = {Hasse, Helmut},
     TITLE = {Theorie der relativ-zyklischen algebraischen
              {F}unktionenk\"{o}rper, insbesondere bei endlichem
              {K}onstantenk\"{o}rper},
   JOURNAL = {J. Reine Angew. Math.},
  FJOURNAL = {Journal f\"{u}r die Reine und Angewandte Mathematik. [Crelle's
              Journal]},
    VOLUME = {172},
      YEAR = {1935},
     PAGES = {37--54},
      ISSN = {0075-4102},
   MRCLASS = {DML},
  MRNUMBER = {1581435},
       DOI = {10.1515/crll.1935.172.37},
       URL = {https://doi-org.scd-rproxy.u-strasbg.fr/10.1515/crll.1935.172.37},
}

@article {homma_aut,
    AUTHOR = {Homma, Masaaki},
     TITLE = {Automorphisms of prime order of curves},
   JOURNAL = {Manuscripta Math.},
  FJOURNAL = {Manuscripta Mathematica},
    VOLUME = {33},
      YEAR = {1980/81},
    NUMBER = {1},
     PAGES = {99--109},
      ISSN = {0025-2611},
   MRCLASS = {14H30},
  MRNUMBER = {596381},
MRREVIEWER = {R. F. Lax},
       DOI = {10.1007/BF01298341},
       URL = {https://doi-org.scd-rproxy.u-strasbg.fr/10.1007/BF01298341},
}

@article {kraus,
    AUTHOR = {Kraus, Alain},
     TITLE = {Sur le d\'{e}faut de semi-stabilit\'{e} des courbes elliptiques \`a
              r\'{e}duction additive},
   JOURNAL = {Manuscripta Math.},
  FJOURNAL = {Manuscripta Mathematica},
    VOLUME = {69},
      YEAR = {1990},
    NUMBER = {4},
     PAGES = {353--385},
      ISSN = {0025-2611},
   MRCLASS = {11G07},
  MRNUMBER = {1080288},
MRREVIEWER = {Philippe Satg\'{e}},
       DOI = {10.1007/BF02567933},
       URL = {https://doi-org.scd-rproxy.u-strasbg.fr/10.1007/BF02567933},
}

@article {kobayashi,
    AUTHOR = {Kobayashi, Shinichi},
     TITLE = {The local root number of elliptic curves with wild
              ramification},
   JOURNAL = {Math. Ann.},
  FJOURNAL = {Mathematische Annalen},
    VOLUME = {323},
      YEAR = {2002},
    NUMBER = {3},
     PAGES = {609--623},
      ISSN = {0025-5831},
   MRCLASS = {11G07 (11G40 11R32)},
  MRNUMBER = {1923699},
MRREVIEWER = {Anupam Saikia},
       DOI = {10.1007/s002080200318},
       URL = {https://doi-org.scd-rproxy.u-strasbg.fr/10.1007/s002080200318},
}

@misc{lmfdb,
  shorthand    = {LMFDB},
  author       = {{The LMFDB Collaboration}},
  title        = {The {L}-functions and Modular Forms Database},
  howpublished = {\url{http://www.lmfdb.org}},
  year         = {2021},
  note         = {[Online; accessed 29 January 2021]},
}

@article {liu_stable_red,
    AUTHOR = {Liu, Qing},
     TITLE = {Courbes stables de genre {$2$} et leur sch\'{e}ma de modules},
   JOURNAL = {Math. Ann.},
  FJOURNAL = {Mathematische Annalen},
    VOLUME = {295},
      YEAR = {1993},
    NUMBER = {2},
     PAGES = {201--222},
      ISSN = {0025-5831},
   MRCLASS = {14D20 (14D25 14H25)},
  MRNUMBER = {1202389},
MRREVIEWER = {Luca Chiantini},
       DOI = {10.1007/BF01444884},
       URL = {https://doi-org.scd-rproxy.u-strasbg.fr/10.1007/BF01444884},
}

@Article{liu_g2_disc_cond,
 Author = {Qing {Liu}},
 Title = {{Conducteur et discriminant minimal de courbes de genre 2}},
 FJournal = {{Compositio Mathematica}},
 Journal = {{Compos. Math.}},
 ISSN = {0010-437X; 1570-5846/e},
 Volume = {94},
 Number = {1},
 Pages = {51--79},
 Year = {1994},
 Publisher = {Cambridge University Press, Cambridge; London Mathematical Society, London},
 MSC2010 = {14H45 14F45},
 Zbl = {0837.14023}
}

@article {liu_algo,
    AUTHOR = {Liu, Qing},
     TITLE = {Mod\`eles minimaux des courbes de genre deux},
   JOURNAL = {J. Reine Angew. Math.},
  FJOURNAL = {Journal f\"{u}r die Reine und Angewandte Mathematik. [Crelle's
              Journal]},
    VOLUME = {453},
      YEAR = {1994},
     PAGES = {137--164},
      ISSN = {0075-4102},
   MRCLASS = {14E30 (14H10)},
  MRNUMBER = {1285783},
MRREVIEWER = {Giorgio Bolondi},
       DOI = {10.1515/crll.1994.453.137},
       URL = {https://doi-org.scd-rproxy.u-strasbg.fr/10.1515/crll.1994.453.137},
}

@article {liu_models,
    AUTHOR = {Liu, Qing},
     TITLE = {Mod\`eles entiers des courbes hyperelliptiques sur un corps de
              valuation discr\`ete},
   JOURNAL = {Trans. Amer. Math. Soc.},
  FJOURNAL = {Transactions of the American Mathematical Society},
    VOLUME = {348},
      YEAR = {1996},
    NUMBER = {11},
     PAGES = {4577--4610},
      ISSN = {0002-9947},
   MRCLASS = {11G20 (11G07 14G20)},
  MRNUMBER = {1363944},
MRREVIEWER = {Dino J. Lorenzini},
       DOI = {10.1090/S0002-9947-96-01684-4},
       URL = {https://doi-org.scd-rproxy.u-strasbg.fr/10.1090/S0002-9947-96-01684-4},
}

@book {liu_book,
    AUTHOR = {Liu, Qing},
     TITLE = {Algebraic geometry and arithmetic curves},
    SERIES = {Oxford Graduate Texts in Mathematics},
    VOLUME = {6},
      NOTE = {Translated from the French by Reinie Ern\'{e},
              Oxford Science Publications},
 PUBLISHER = {Oxford University Press, Oxford},
      YEAR = {2002},
     PAGES = {xvi+576},
      ISBN = {0-19-850284-2},
   MRCLASS = {14-01 (11G30 14A05 14A15 14Gxx 14Hxx)},
  MRNUMBER = {1917232},
MRREVIEWER = {C\'{\i}cero Carvalho},
}

@article {liu_neron,
    AUTHOR = {Liu, Qing and Tong, Jilong},
     TITLE = {N\'{e}ron models of algebraic curves},
   JOURNAL = {Trans. Amer. Math. Soc.},
  FJOURNAL = {Transactions of the American Mathematical Society},
    VOLUME = {368},
      YEAR = {2016},
    NUMBER = {10},
     PAGES = {7019--7043},
      ISSN = {0002-9947},
   MRCLASS = {14H25 (11G35 14G20 14G40)},
  MRNUMBER = {3471084},
MRREVIEWER = {Adolfo Quir\'{o}s},
       DOI = {10.1090/tran/6642},
       URL = {https://doi-org.scd-rproxy.u-strasbg.fr/10.1090/tran/6642},
}

@book {mumford_tata2,
    AUTHOR = {Mumford, David},
     TITLE = {Tata lectures on theta. {II}},
    SERIES = {Progress in Mathematics},
    VOLUME = {43},
      NOTE = {Jacobian theta functions and differential equations,
              With the collaboration of C. Musili, M. Nori, E. Previato, M.
              Stillman and H. Umemura},
 PUBLISHER = {Birkh\"{a}user Boston, Inc., Boston, MA},
      YEAR = {1984},
     PAGES = {xiv+272},
      ISBN = {0-8176-3110-0},
   MRCLASS = {14K25 (14H40 32G20)},
  MRNUMBER = {742776},
MRREVIEWER = {M. Kh. Gizatullin},
       DOI = {10.1007/978-0-8176-4578-6},
       URL = {https://doi-org.scd-rproxy.u-strasbg.fr/10.1007/978-0-8176-4578-6},
}

@article {namikawa_ueno,
    AUTHOR = {Namikawa, Yukihiko and Ueno, Kenji},
     TITLE = {The complete classification of fibres in pencils of curves of
              genus two},
   JOURNAL = {Manuscripta Math.},
  FJOURNAL = {Manuscripta Mathematica},
    VOLUME = {9},
      YEAR = {1973},
     PAGES = {143--186},
      ISSN = {0025-2611},
   MRCLASS = {14D05 (14H20 32J15)},
  MRNUMBER = {369362},
MRREVIEWER = {W.-D. Geyer},
       DOI = {10.1007/BF01297652},
       URL = {https://doi-org.scd-rproxy.u-strasbg.fr/10.1007/BF01297652},
}

@book {neukirch,
    AUTHOR = {Neukirch, J\"urgen},
     TITLE = {Algebraic number theory},
    SERIES = {Grundlehren Math. Wiss.},
    VOLUME = {322},
 PUBLISHER = {Springer-Verlag, Berlin},
      YEAR = {1999},
     PAGES = {xviii+571},
   MRCLASS = {11Rxx (11-02 11S15 11S31 14C40)},
  MRNUMBER = {1697859},
MRREVIEWER = {Cornelius Greither},
       DOI = {10.1007/978-3-662-03983-0},
}

@incollection {rohrlich,
    AUTHOR = {Rohrlich, David E.},
     TITLE = {{E}lliptic curves and the {W}eil-{D}eligne group},
 BOOKTITLE = {Elliptic curves and related topics},
    SERIES = {CRM Proc. Lecture Notes},
    VOLUME = {4},
     PAGES = {125--157},
 PUBLISHER = {Amer. Math. Soc., Providence, RI},
      YEAR = {1994},
   MRCLASS = {11G07 (11F80 11F85)},
  MRNUMBER = {1260960},
MRREVIEWER = {Henri Darmon},
}

@article {rohrlich_formulas,
    AUTHOR = {Rohrlich, David E.},
     TITLE = {Galois theory, elliptic curves, and root numbers},
   JOURNAL = {Compositio Math.},
  FJOURNAL = {Compositio Mathematica},
    VOLUME = {100},
      YEAR = {1996},
    NUMBER = {3},
     PAGES = {311--349},
   MRCLASS = {11G05 (11F80 11G07 11G40 11R32)},
  MRNUMBER = {1387669},
MRREVIEWER = {Kenneth Kramer},
}

@article {roquette,
    AUTHOR = {Roquette, Peter},
     TITLE = {Absch\"{a}tzung der {A}utomorphismenanzahl von {F}unktionenk\"{o}rpern
              bei {P}rimzahlcharakteristik},
   JOURNAL = {Math. Z.},
  FJOURNAL = {Mathematische Zeitschrift},
    VOLUME = {117},
      YEAR = {1970},
     PAGES = {157--163},
      ISSN = {0025-5874},
   MRCLASS = {14.35},
  MRNUMBER = {279100},
MRREVIEWER = {F. Oort},
       DOI = {10.1007/BF01109838},
       URL = {https://doi-org.scd-rproxy.u-strasbg.fr/10.1007/BF01109838},
}

@article {sabitova_root,
    AUTHOR = {Sabitova, Maria},
     TITLE = {Root numbers of abelian varieties},
   JOURNAL = {Trans. Amer. Math. Soc.},
  FJOURNAL = {Transactions of the American Mathematical Society},
    VOLUME = {359},
      YEAR = {2007},
    NUMBER = {9},
     PAGES = {4259--4284},
%      ISSN = {0002-9947},
   MRCLASS = {11G10 (11F80 11G40 11R32)},
  MRNUMBER = {2309184},
MRREVIEWER = {Jae-Hyun Yang},
       DOI = {10.1090/S0002-9947-07-04148-7},
%       URL = {https://doi-org.scd-rproxy.u-strasbg.fr/10.1090/S0002-9947-07-04148-7},
}

@misc{serre_2torsion,
  title={Rigidité du foncteur de Jacobi d’échelon {$n\geqslant3$}},
  %note = {talk:17},
  author={Serre, Jean-Pierre},
  %year={1960-1961},
  howpublished = {Appendix to A. Grothendieck, Techniques de construction en géométrie analytique, X. Construction de l'espace de Teichmüller, Séminaire Henri Cartan, 1960/61, no. 17},
   shorthand = {Ser61},
}

@book {serre_localfields,
    AUTHOR = {Serre, Jean-Pierre},
     TITLE = {Local fields},
    SERIES = {Graduate Texts in Mathematics},
    VOLUME = {67},
      NOTE = {Translated from the French by Marvin Jay Greenberg},
 PUBLISHER = {Springer-Verlag, New York-Berlin},
      YEAR = {1979},
     PAGES = {viii+241},
   MRCLASS = {12Bxx},
  MRNUMBER = {554237},
}

@article {serre_tate,
    AUTHOR = {Serre, Jean-Pierre and Tate, John},
     TITLE = {Good reduction of abelian varieties},
   JOURNAL = {Ann. of Math. (2)},
  FJOURNAL = {Annals of Mathematics. Second Series},
    VOLUME = {88},
      YEAR = {1968},
     PAGES = {492--517},
   MRCLASS = {14.51},
  MRNUMBER = {0236190},
MRREVIEWER = {M. J. Greenberg},
}

@article {silverberg_zarhin,
    AUTHOR = {Silverberg, A. and Zarhin, Yu. G.},
     TITLE = {Inertia groups and abelian surfaces},
   JOURNAL = {J. Number Theory},
  FJOURNAL = {Journal of Number Theory},
    VOLUME = {110},
      YEAR = {2005},
    NUMBER = {1},
     PAGES = {178--198},
      ISSN = {0022-314X},
   MRCLASS = {11G10},
  MRNUMBER = {2114680},
       DOI = {10.1016/j.jnt.2004.05.015},
       URL = {https://doi-org.scd-rproxy.u-strasbg.fr/10.1016/j.jnt.2004.05.015},
}

@book {stichtenoth_ff,
    AUTHOR = {Stichtenoth, Henning},
     TITLE = {Algebraic function fields and codes},
    SERIES = {Graduate Texts in Mathematics},
    VOLUME = {254},
   EDITION = {Second},
 PUBLISHER = {Springer-Verlag, Berlin},
      YEAR = {2009},
     PAGES = {xiv+355},
      ISBN = {978-3-540-76877-7},
   MRCLASS = {14H05 (11R58 11T71 14G15 14G50 94B27)},
  MRNUMBER = {2464941},
}

@article {yelton_4-tors,
    AUTHOR = {Yelton, Jeffrey},
     TITLE = {Images of 2-adic representations associated to hyperelliptic
              {J}acobians},
   JOURNAL = {J. Number Theory},
  FJOURNAL = {Journal of Number Theory},
    VOLUME = {151},
      YEAR = {2015},
     PAGES = {7--17},
      ISSN = {0022-314X},
   MRCLASS = {14H40 (11F50 11G05 11G30 12F10)},
  MRNUMBER = {3314198},
MRREVIEWER = {Rolf Berndt},
       DOI = {10.1016/j.jnt.2014.10.020},
       URL = {https://doi-org.scd-rproxy.u-strasbg.fr/10.1016/j.jnt.2014.10.020},
}

\end{document}